\documentclass[12pt,draftclsnofoot,onecolumn]{IEEEtran}

\usepackage{amsmath,amssymb,enumerate,color,bbm,graphicx,stmaryrd}
\usepackage[latin1]{inputenc}
\usepackage[english]{babel}
\usepackage{graphicx}
\usepackage{multirow}

\definecolor{gris75}{gray}{0.3}
\newtheorem{remark}{Remark}
\newtheorem{prop}{Proposition}

\newtheorem{theo}{Theorem}
\newtheorem{lemma}{Lemma}

\newtheorem{coro}{Corollary}
\newtheorem{assum}{Assumption}
\newtheorem{cond}{Condition}

\newcounter{hypA}

\newcounter{hyp}

\newcommand{\un}{{\bs 1}}

\newcommand{\eqdef}{:=}

\newcommand{\bs}{\boldsymbol}

\newcommand{\sH}{{\mathsf H}}
\newcommand{\sE}{{\mathsf E}}
\newcommand{\sV}{{\mathsf V}}
\newcommand{\sd}{{\mathsf d}}

\newcommand{\cT}{{\mathcal T}}

\newcommand{\cP}{{\mathcal P}}

\newcommand{\cA}{{\mathcal A}}
\newcommand{\cK}{{\mathcal K}}
\newcommand{\cL}{{\mathcal L}}

\newcommand{\cF}{{\mathcal F}}
\newcommand{\cN}{{\mathcal N}}

\newcommand{\cV}{{\cal V}}

\newcommand{\cW}{{\cal W}}

\newcommand{\cS}{{\mathcal S}}

\newcommand{\bP}{{\mathbb P}}
\newcommand{\bR}{{\mathbb R}}

\newcommand{\bE}{{\mathbb E}}

\newcommand{\dG}{\partial G}

\newcommand{\la}{\langle}
\newcommand{\ra}{\rangle}

\newcommand{\bth}{{\bs \theta}}

\newcommand{\bY}{{\bs Y}}
\newcommand{\by}{{\bs y}}
\newcommand{\bsZ}{{\bs Z}}
\newcommand{\ath}{\langle \bth\rangle}
\newcommand{\thn}{{\bs \theta}_{n}}
\newcommand{\othn}{{\bs \theta}_{\bot,n}}
\newcommand{\oth}{{\bs \theta}_{\bot}}

\newcommand{\thnmu}{{\bs \theta}_{n-1}}
\newcommand{\othnmu}{{\bs \theta}_{\bot,n-1}}
\newcommand{\thkmu}{{\bs \theta}_{k-1}}
\newcommand{\othkmu}{{\bs \theta}_{\bot,k-1}}

\newcommand{\athn}{\langle\thn\rangle}

\newcommand{\athnmu}{\langle\thnmu\rangle}

\newcommand{\Jo}{{J_\bot}}

\newcommand{\nmu}{_{n-1}}

\title{Convergence~of a Multi-Agent Projected Stochastic Gradient Algorithm for Non-Convex Optimization}

\author{ Pascal Bianchi,~\IEEEmembership{Member,~IEEE}, \ J\'er\'emie
  Jakubowicz~\IEEEmembership{Member,~IEEE} \thanks{P. Bianchi is with Institut
    Mines - T\'el\'ecom - Telecom ParisTech - CNRS/LTCI, 75634 Paris Cedex 13, France
    (e-mail: bianchi@telecom-paristech.fr). J. Jakubowicz is with Institut Mines-T\'el\'ecom - Telecom SudParis -
CNRS/SAMOVAR, 91000 Evry, France (e-mail: jeremie.jakubowicz@telecom-sudparis.fr).
This work has been supported
    by ANR Program C-FLAM.}  }

\begin{document}

\maketitle

\begin{abstract}
  We introduce a new framework for the convergence analysis of a class of
  distributed constrained non-convex optimization algorithms in multi-agent
  systems.  The aim is to search for local minimizers of a non-convex objective
  function which is supposed to be a sum of local utility functions of the
  agents.  The algorithm under study consists of two steps: a local stochastic
  gradient descent at each agent and a gossip step that drives the network of
  agents to a consensus.  Under the assumption of decreasing stepsize, it is
  proved that consensus is asymptotically achieved in the network and that the
  algorithm converges to the set of Karush-Kuhn-Tucker points.  As an important
  feature, the algorithm does not require the double-stochasticity of the gossip
  matrices.  It is in particular suitable for use in a natural broadcast
  scenario for which no feedback messages between agents are required.  It is
  proved that our results also holds if the number of communications in the
  network per unit of time vanishes at moderate speed as time increases,
  allowing potential savings of the network's energy.  Applications to power
  allocation in wireless ad-hoc networks are discussed.  Finally, we provide
  numerical results which sustain our claims.
\end{abstract}

\section{Introduction}
\label{sec:intro}

Stochastic gradient descent is a widely used procedure for finding critical
points of an unknown function $f$~\cite{robbins:monro:1951}.
Formally, it can be summarized as an iterative scheme of the form $\theta_{n+1}
= \theta_n + \gamma_{n+1} (-\nabla f(\theta_n)+\xi_{n+1})$ where $\nabla$ is the
gradient operator and where $\xi_{n+1}$ represents a random
perturbation. Relevant selection of the step size $\gamma_n$ ensures that, for a
well behaved function $f$, sequence $(\theta_n)_{n\geq 0}$ will eventually
converge to a critical point.

In this paper, we investigate a distributed optimization problem which is of
practical interest in many multi-agent contexts such as parallel
computing~\cite{bertsekas:1997}, statistical
estimation~\cite{schizas:2008,schizas:2009,barbarossa:2007,ram:veeravalli:nedic:TAC:2010},
robotics~\cite{BuCoMa08} or wireless
networks~\cite{ram:veeravalli:nedic:INFOCOM:2009}.  Consider a network of $N$
agents.  To each agent $i=1,\dots, N$, we associate a possibly non-convex
continuously differentiable utility function $f_i:\bR^d\to\bR$ where $d\geq 1$.
Let $G\subset \bR^d$ be a nonempty compact convex subset. We address the the
following optimization problem:
\begin{equation}
  \label{eq:pb}
  \min_{\theta\in G} \sum_{i=1}^N f_i(\theta)\ .
\end{equation}
The set $G$ is assumed to be known by all agents.  However, a given agent $i$
does not know the utility functions $f_j$'s of other agents $j\neq
i$. Cooperation between agents is therefore needed to find minimizers
of~(\ref{eq:pb}).  Moreover, any utility function $f_i$ may be unperfectly
observed by agent $i$ itself, due to the presence of random observation
noise. We thus address the framework of distributed \emph{stochastic
  approximation}.

The literature contains at least two different cooperation approaches for
solving~(\ref{eq:pb}). The so-called \emph{incremental} approach is used
by~\cite{rabbat:nowak:2004, nedic:bertsekas:jopt-2001,
  rabbat:nowak:sac-2005,johansson:rabi:johansson:jopt-2009,
  ram:nedic:veeravalli:jopt-2009}: a message containing an estimate of the
desired minimizer iteratively travels all over the network.  At any instant, the
agent which is in possession of the message updates its own estimate and adds
its own contribution, based on its local observation.  Incremental algorithms
generally require the message to go through a Hamiltonian cycle in the
network. Finding such a path is known to be a NP complete problem.  Relaxations
of the Hamiltonian cycle requirement have been proposed: for instance,
\cite{nedic:bertsekas:jopt-2001} only requires that an agent communicates with
another agent randomly selected in the network (not necessarily in its
neighborhood) according to the uniform distribution.  However, substantial
routing is still needed. In~\cite{lu:tang:regier:bow:tac-2011},
problem~(\ref{eq:pb}) is solved using a different approach, assuming that agents
perfectly observe their utility functions and know also the utility functions of
their neighbors.

This paper focuses on another cooperation approach based on \emph{average
  consensus} techniques, see
references~\cite{degroot:1974,tsitsiklis:phd-1984}. In this context, each agent
maintains its own estimate.  Agents separately run local gradient algorithms and
simultaneously communicate in order to eventually reach an agreement over the
whole network on the value of the minimizer.  Communicating agents combine their
local estimates in a linear fashion: a receiver computes a weighted average
between its own estimate and the ones which have been transmitted by its
neighbors.  Such combining techniques are often refered to as \emph{gossip}
methods.

The idea underlying the algorithm of interest in this paper is not new.  Its
roots can be found in 
\cite{tsitsiklis:phd-1984,tsitsiklis:bertsekas:athans:tac-1986} where a network
of processors seeks to optimize some objective function \emph{known} by all
agents (possibly up to some additive noise). More recently, numerous works
extended this kind of algorithm to more involved multi-agent scenarios,
see~\cite{nedic:ozdaglar:tac-2009,nedic:ozdaglar:parrilo:tac-2010,ram:nedic:veeravalli:jota-2010,nedic:tac-2011,kar:2010,stankovic:stankovic:stipanovic:tac-2011}
for a non exhaustive list.  Multi-agent systems are indeed more difficult to
deal with, because individual agents does not know the global objective function
to be minimized. Reference \cite{nedic:ozdaglar:tac-2009} addresses the problem
of \emph{unconstrained} optimization, assuming \emph{convex} but not necessarily
differentiable utility functions.  Convergence to a global minimizer is
established assuming that utility functions have bounded (sub)gradients. Let us
also mention \cite{stankovic:stankovic:stipanovic:tac-2011} which focuses on the
case of \emph{quadratic} objective functions. Unconstrained optimization is also
investigated in~\cite{nous:eusipco} assuming differentiable but non necessarily
convex utility functions and relaxing boundedness conditions on the
gradients. Convergence to a critical point of the objective function is proved
and the asymptotic performance is evaluated under the form of a central limit
theorem.
In \cite{nedic:ozdaglar:parrilo:tac-2010}, the problem of \emph{constrained}
distributed optimization is addressed.  Convergence to an optimal consensus is
proved when each utility function $f_i$ is assumed convex and perfectly known by
agent $i$.  These results are extended in
\cite{ram:nedic:veeravalli:jota-2010} to the stochastic descent case
\emph{i.e.}, when the observation of utility functions is perturbed by a random
noise.

In each of these works, the gossip communication scheme can be represented by a
sequence of matrices $(W_n)_{n\geq 1}$ of size $N\times N$, where the $(i,j)$th
component of $W_n$ is the weight given by agent $i$ to the message received from
$j$ at time $n$, and is equal to zero in case agent $i$ receives no message from
$j$.  In most works (see for instance \cite{nedic:ozdaglar:tac-2009,
  nedic:ozdaglar:parrilo:tac-2010,
  ram:nedic:veeravalli:jota-2010,nous:eusipco}), matrices $W_n$ are assumed
\emph{doubly stochastic}, meaning that $W_n^T\un=W_n\un=\un$ where $\un$ the $N\times 1$ vector whose components are all equal to one and where $^T$ denotes
transposition. Although row-stochasticity ($W_n\un=\un$) is rather easy to
ensure in practice, column-stochasticity ($W_n^T\un=\un$) implies more stringent
restrictions on the communication protocol. For instance, in~\cite{boyd:2006},
each one-way transmission from an agent $i$ to another agent $j$ requires at the
same time a feedback link from $j$ to $i$.  Double stochasticity prevents one
from using natural broadcast schemes, in which a given agent may transmit its
local estimate to \emph{all} its neighbors without expecting any immediate
feedback~\cite{aysal:2009}.
Very recently, \cite{nedic:tac-2011} made a major step forward, getting rid of
the column stochasticity condition, and thus opening the road to broadcast-based
constrained distributed optimization algorithms. It is worth noting however that
the algorithm of \cite{nedic:tac-2011} is such that only receiving agents update
their estimates. Otherwise stated, an agent deletes its local observations as
long as it is not the recipient of a message.
Moreover, except perhaps in some special network topologies, the algorithm
of~\cite{nedic:tac-2011} strongly relies on a specific choice of the stepsize.
In particular, a necessary condition for the convergence to the desired
consensus is that the stepsize vanishes at speed $1/n$.  However, in practice,
it is often desirable to have a leeway on the choice of the stepsize to avoid
slow convergence issues.

\subsection*{Contributions}

In this paper, we address the optimization problem~(\ref{eq:pb}) using a
distributed projected stochastic gradient algorithm involving random gossip
between agents and decreasing stepsize.
\begin{itemize}
\item Unlike previous works, utility functions are allowed to be
  \emph{non-convex}. We introduce a new framework for the analysis of a general
  class of distributed optimization algorithm, which does not rely on convexity
  properties of the utility functions. Instead, our approach relies on recent
  results of reference~\cite{benaim:2005} about perturbed differential inclusions. 
  Under a set of assumptions made clear in the next section, we establish that, almost surely,
  the sequence of estimates of any agent shadows the behavior of a differential variational
  inequality, and eventually converges to the set of Karush-Kuhn-Tucker (KKT) points of~(\ref{eq:pb}).

\item Our assumptions encompass the case of non-doubly stochastic matrices $W_n$ and, as a particular case, the natural broadcast scheme
  of reference~\cite{aysal:2009}.  Our proofs reveal that, loosely speaking, the
  relaxation of column stochasticity brings a ``noise-like'' term in the
  algorithm dynamics, but which is not powerful enough to prevent convergence to
  the KKT points.

\item We show that our convergence result still holds in case the number of
  communications in the network per unit of time vanishes at moderate speed as
  time increases.
\end{itemize}
As an illustration, we apply our results to the problem of power allocation in
the wireless interference channel.  \medskip


The paper is organized as follows. Section \ref{sec:model} introduces the
distributed algorithm and the main assumptions on the network and the
observation model. The main result is stated in Section~\ref{sec:constrained}.
Section~\ref{sec:proofcvc} is devoted to its proof.  We discuss applications to
power allocation in Section~\ref{sec:appli}.  Section~\ref{sec:simu} describes
some standard communication schemes in more details, and provides numerical results.

\section{The Distributed Algorithm}
\label{sec:model}

\subsection{Description of the Algorithm}

After an
\noindent{\tt[Initialization step]} where each agent starts at a given
$\theta_{0,i}$, each node $i$ generates a stochastic process
$(\theta_{n,i})_{n\geq 1}$ in $\bR^d$
using a two-step iterative algorithm:\\
\noindent {\tt [Local step]} Node $i$ generates at time $n$ a temporary estimate
$\tilde \theta_{n,i}$ given by
\begin{equation}
  \label{eq:tempupdate}
  \tilde \theta_{n,i}= P_{G}\left[\theta_{n-1,i} + \gamma_{n}Y_{n,i}\right]\ ,
\end{equation}
where $\gamma_n$ is a deterministic positive step size, $Y_{n,i}$ is a random
variable, and $P_{G}$ represents the projection operator onto the set $G$.
Random variable $Y_{n,i}$ is to be interpreted as a perturbed version of the
negative gradient of $f_i$ at point $\theta_{n-1,i}$. As will be made clear by
Assumption~\ref{hyp:model}(e) below, it is convenient to think of $Y_{n,i}$ as
$Y_{n,i}=-\nabla f_i(\theta_{n-1,i})+\delta M_{n,i}$ where $\delta M_{n,i}$ is a
martingale difference noise which stands for the random perturbation.

\noindent {\tt [Gossip step]} Node $i$ is able to observe the values
$\tilde\theta_{n,j}$ of some other $j$'s and computes the weighted average:
$$
\theta_{n,i}=\sum_{j=1}^N w_n(i,j)\,\tilde \theta_{n,j}
$$
where for any $i$, $\sum_{j=1}^Nw_n(i,j)=1$. In the sequel, we define the $N\times  N$
matrix $W_n := [w_n(i,j)]_{i,j=1\cdots N}$.

Define the random vectors $\thn$ and $\bY_n$ as
$\thn:=(\theta_{n,1}^T,\dots,\theta_{n,N}^T)^T$ and $\bY_n = (Y_{n,1}^T,\dots,
Y_{n,N}^T)^T$.  The algorithm reduces to:
\begin{equation}
  \thn = (W_n\otimes I_d)P_{G^N}\left[\thnmu+\gamma_{n}\bY_n\right]
  \label{eq:algo} 
\end{equation}
where $\otimes$ denotes the Kronecker product, $I_d$ is the $d\times  d$ identity
matrix and $P_{G^N}$ is the projector onto the $N$th order product set
$G^N:=G\times \cdots\times  G$.

\subsection{Observation and Network Models}
\label{sec:assumptions}



Random processes $(\bY_n,W_n)_{n\geq 1}$ are defined on a measurable space
equipped with a probability~$\bP$. Notation $\bE$ represents the corresponding
expectation.  For any $n\geq 1$, we introduce the $\sigma$-field $\cF_n =
\sigma({\bs \theta}_0,\bY_{1},\dots,\bY_{n},W_{1},\dots,W_{n})$. The
distribution of the random vector $\bY_{n+1}$ conditioned on $\cF_n$ is assumed
to be such that:
\begin{equation}
\label{eq:lawYn}
\bP\left( \bY_{n+1} \in A \, |\, \cF_n\right) = \mu_{\thn}(A)
\end{equation}
for any measurable set $A$, where $\left(\mu_\bth\right)_{\bth\in \bR^{dN}}$ is
a given family of probability measures on $\bR^{dN}$. 
We denote by $|x|$ the Euclidean norm of a vector $x$. We denote by $E^c$ the
complementary set of any set $E$. Notation $x\lor y$ stands for $\max(x,y)$.
\begin{assum}
  \label{hyp:model} The following conditions hold:
  \begin{enumerate}
  \item[a)] $(W_n)_{n\geq 1}$ is a sequence of matrix-valued random variables with non-negative components~s.t.
    \begin{itemize}
    \item $W_n$ is row stochastic for any $n$: $W_n\un=\un$,
    \item $\bE(W_n)$ is column stochastic for any $n$: $\un^T\bE(W_n)=\un^T$.
    \end{itemize}
  \item[b)] The spectral norm $\rho_n$ of matrix
    $\bE(W_n^T(I_N-\un\un^T/N)W_n)$ satisfies:
    \begin{equation}
      \lim_{n\to\infty}n(1-\rho_n)=+\infty\ .
      \label{eq:limrhon}
    \end{equation}
  \item[c)] Conditionaly to $\cF_n$, $W_{n+1}$ and $\bs Y_{n+1}$ are
    independent. Moreover, the probability distribution of $\bs Y_{n+1}$
    conditionaly to $\cF_n$ depends on $\bs\theta_n$ only. It is denoted
    $\mu_{\theta_n}$.
  \item[d)] For any $i=1,\dots, N$, $f_i$ is continuously differentiable.
  \item[e)] For any $n\geq 1$,
$$
\bE[\bY_{n+1}|\cF_n] = -(\nabla f_1(\theta_{n,1})^T,\cdots,\nabla f_N(\theta_{n,N})^T)^T\ .
$$
\item[f)] $\sup_{\bth\in G^N} \int |\bs y|^2 \sd\mu_{\bth}(\bs y)<\infty$.
\end{enumerate}\end{assum}
We now discuss the above Assumption.  Conditions \ref{hyp:model}(a) and
\ref{hyp:model}(b) summarize our assumptions on matrices $W_n$ that is, on the
communication scheme used in the network. Following the work
of reference~\cite{boyd:2006}, random gossip is assumed in this paper. Each matrix $W_n$
must be row stochastic, this means that each agent $i=1,\cdots,N$ must compute a
weighted \emph{average} $\sum_j w_n(i,j)=1$.  Note that a quite classical
condition in the literature is to further assume that $W_n$ is column-stochastic
for any $n$
\cite{nedic:ozdaglar:tac-2009,nedic:ozdaglar:parrilo:tac-2010,ram:nedic:veeravalli:jota-2010,nous:eusipco}.
Column stochasticity inevitably goes with some restrictions on the communication
protocol as discussed in Section~\ref{sec:intro}. Here, our assumption is
weaker. We only require that $W_n$ is column stochastic \emph{on average}. This
is for instance the case in the natural broadcast scheme of reference~\cite{aysal:2009}
which will be discussed in the Section~\ref{sec:gossip}.  
Assumption~\ref{hyp:model}(b)
is a connectivity condition of the underlying network graph which
will be discussed in more details in Section~\ref{sec:gossip}.


Assumptions~\ref{hyp:model}(c-e) are related to the observation model.
Assumption~\ref{hyp:model}(c) states three different things. First, random variables $W_{n+1}$ and
$\bY_{n+1}$ are independent conditionally on the past. Second,
$(W_n)_{n\geq 1}$ form an independent sequence (not necessarily identically
distributed).  Finally, the distribution of $\bY_{n+1}$ conditionally on the past is as in~(\ref{eq:lawYn}).
Assumption~\ref{hyp:model}(e) means that each $Y_{n,i}$ can be
interpreted as a noisy version of $-\nabla f_i(\theta_{n-1,i})$. The
distribution of the random additive perturbation $Y_{n,i}-(-\nabla
f_i(\theta_{n-1,i}))$ is likely to depend on the past through the value of
$\thnmu$, but has a zero mean for any given value of $\thnmu$.
By Markov's inequality, Assumption~\ref{hyp:model}(f) implies that $(\mu_\bth)_{\bth\in G^N}$ is tight \emph{i.e.},
for any $\epsilon>0$, there exists a compact set $\cK\subset\bR^{dN}$ such that $\sup_{\bth\in G^N}\mu_{\bth}({\cK}^c)<\epsilon$.
\begin{assum} \label{hyp:step}
  \begin{enumerate}\item[a)] The deterministic sequence $(\gamma_n)_{n\geq 1}$
    is positive and such that $\sum_n\gamma_n=\infty$.
  \item[b)] There exists $\alpha>1/2$ such that:
    \begin{align}
      &\lim_{n\to\infty} n^{\alpha}\gamma_n = 0 \label{eq:condgamma}\\
      &\liminf_{n\to\infty} \frac{1-\rho_n}{n^{\alpha}\gamma_n}>0 \
      . \label{eq:condrho}
    \end{align}
  \end{enumerate}
\end{assum}
Note that, (\ref{eq:condgamma}) implies (but is not equivalent to)
$\sum_n\gamma_n^2 <\infty$, which is a rather usual assumption in the framework
of decreasing step size stochastic algorithms~\cite{kushner:2003}. In order to
have some insights on~(\ref{eq:condrho}), first consider the case where the
matrices $(W_n)_{n\geq 1}$ form an i.i.d. sequence \emph{i.e.}, the spectral
radius $\rho:=\rho_n$ does not depend on $n$.  Then both
conditions~(\ref{eq:limrhon}) and~(\ref{eq:condrho}) are satisfied if and only
if:
\begin{equation}
  \label{eq:rhoinfun}
  \rho <1 \ .
\end{equation}
Nevertheless, matrices $(W_n)_{n\geq 1}$ do not need to be i.i.d. An interesting
example is when matrix $W_n$ is likely to be equal to identity with a
probability that tends to one as $n\to \infty$. From a communication point of
view, this means that the exchange of information between agents becomes rare as
$n\to\infty$. This context is especially interesting in case of wireless
networks, where it is often required to limit as much as possible the
communication overhead. Let us emphasize that if $W_n$ is assumed i.i.d with
$\rho<1$, then the usual assumptions $\sum \gamma_n = +\infty$ and $\sum
\gamma_n^2<\infty$ can be substituted to Assumption~\ref{hyp:step}. Assumption~\ref{hyp:step} is designed to take into account the non-stationnarity of $W_n$.

Consider for instance the case where $1-\rho_n=a/n^\eta$ and
$\gamma_n=\gamma_0/n^\xi$ for some constants $a,\gamma_0>0$.  Then, a sufficient
condition for~Assumption~\ref{hyp:step} is:
$$
0\leq \eta < \xi -1/2 \leq 1/2\ .
$$
In particular, $\xi\in (1/2,1]$ and $\eta\in [0,1/2)$.


\subsection{Illustration: Some Examples of Gossip schemes}
\label{sec:gossip}

Here, we focus on two standard communication schemes and give the corresponding
sequence $(W_n)_{n\geq 1}$ for each of them. We refer the reader to
reference~\cite{benezit:thesis} for a more complete picture and for more general
gossip strategies.
We introduce what we shall refer to as the \emph{pairwise} and the
\emph{broadcast} schemes.  The first one can be found in the paper of Boyd et
al.~\cite{boyd:2006} on average consensus while the second is inspired from the
broadcast scheme depicted in~\cite{aysal:2009}.  The network of agents is
represented as a nondirected graph $(\sV,\sE)$ where $\sV$ is the set of $N$
nodes and $\sE$ corresponds to the set edges between nodes.

\subsubsection{Pairwise Gossip}

At time $n$, a single node $i$ wakes up (node $i$ is chosen at random, uniformly
within the set of nodes and independently from the past). Node $i$ randomly
selects a node $j$ among its neighbors in the graph. Node $i$ and $j$ exchange
their temporary estimates $\tilde\theta_{n,i}$ and $\tilde\theta_{n,j}$ and
compute the weighted average $\theta_{n,i}=\theta_{n,j}= \beta\tilde\theta_{n,i}
+ (1-\beta) \tilde\theta_{n,j}$ where $0<\beta <1$. Other nodes $k\notin\{i,j\}$
simply set $\theta_{n,k}=\tilde \theta_{n,k}$. Set $\beta=1/2$ for simplicity.
In this case, the corresponding matrix $W_n$ is given by $W_n =
I_N-(e_i-e_j)(e_i-e_j)^T/2$ where $e_i$ denotes the $i$th vector of the
canonical basis in $\bR^N$. Note that for each $n$, $W_n$ forms an
i.i.d. sequence of \emph{doubly stochastic} matrices.
Assumption~\ref{hyp:model}(a) is obviously satisfied.  Moreover, the spectral
radius $\rho$ of matrix $\bE(W_n^T(I_N-\un\un^T/N)W_n)$
satisfies~(\ref{eq:rhoinfun}) if and only if $(\sV,\sE)$ is a connected graph
(see~\cite{boyd:2006}).

\subsubsection{Broadcast Gossip}

At time $n$, a random node $i$ wakes up. The latter node is supposed to be
chosen at random w.r.t. the uniform distribution on the set of vertices.  Node
$i$ broadcasts its temporary update to all its neighbors. Any neighbor $j$,
computes the weighted average $\theta_{n,j}= \beta\tilde\theta_{n,i} + (1-\beta)
\tilde\theta_{n,j}$. On the other hand, any node $k$ which does not belong to
the neighborhood $\cN_i$ of $i$ (this includes $i$ itself) simply sets
$\theta_{n,k}=\tilde\theta_{n,k}$. Note that, as opposed to the pairwise scheme,
the transmitter node $i$ does not expect any feedback from its neighbors.  It is
straightforward to show that the $(k,\ell)$th component of matrix $W_n$
corresponding to such a scheme writes:
$$
w_n(k,\ell) = \left\{
  \begin{array}[h]{ll}
    1 & \textrm{if } k\notin \cN_i \textrm{ and }k=\ell\\
    \beta  & \textrm{if } k\in\cN_i \textrm{ and }\ell=i\\
    1-\beta & \textrm{if } k\in\cN_i \textrm{ and }k=\ell\\
    0 & \textrm{otherwise.}
  \end{array}\right.
$$
As a matter of fact, the above matrix $W_n$ is not doubly stochastic since
$\un^TW_n\neq \un^T$.  Nevertheless, it is straightfoward to check that
$\un^T\bE(W_n)= \un^T$ (see for instance~\cite{aysal:2009}).  Thus, the sequence
of matrices $(W_n)_{n\geq 1}$ satisfies the Assumption~\ref{hyp:model}(a).  Once
again, straightforward derivations which can be found in~\cite{aysal:2009} show
that the spectral radius $\rho$ satisfies~(\ref{eq:rhoinfun}) if and only if
$(\sV,\sE)$ is a connected graph.

\section{Convergence w.p.1}
\label{sec:constrained}

\subsection{Framework and Assumptions}

We study the case where for any $i=1,\cdots,N$, the set $G$ is determined by a
finite set of $p$ inequality constraints ($1\leq p<\infty$):
\begin{equation}
  \label{eq:G}
  G := \left\{ \theta\in\bR^d\ :\ \forall j=1,\dots,p,\ q_j(\theta)\leq 0\right\}
\end{equation}
for some functions $q_{1},\dots,q_{p}$ which satisfy the following conditions.
For any $\theta\in \bR^d$, we denote by $A(\theta)\subset \{1,\dots,p\}$ the
active set \emph{i.e.}, $A(\theta) = \{j:q_j(\theta)=0,\theta\in G\}$. Denote
by $\dG$ the boundary of $G$.
\begin{assum}
  \label{hyp:KT}
  \begin{enumerate}\item[a)] The set $G$ defined by~(\ref{eq:G}) is nonempty and
    compact.
  \item[b)] For any $j=1,\cdots,p$, $q_j:\bR^d\to\bR$ is a convex function,
    continuously differentiable in a neighborhood of $\dG$.
  \item[c)] For any $\theta\in\dG$, $(\nabla q_j(\theta):j\in A(\theta))$ is a
    linearly independent collection of vectors.
  \end{enumerate}\end{assum}
In other terms, some regularity is imposed on $G$. Moreover, (\ref{eq:G},c) is a
simple \emph{qualification} assumption \cite{rockafellar:wets:1998}: note that
the same set $G$ can be expressed using different constraints; for instance, one
can always duplicate constraints arbitrarily. The qualification assumption says
that it is up to the user to remove redundant constraints.

\subsection{Notations}

Recall that $|x|$ represents the Euclidean norm of a vector $x$. Denote by $\nabla$
the gradient operator. We denote by 
\begin{equation}
  \label{eq:projectors}
  J\eqdef (\un\un^T/N)\otimes I_d \ , \qquad \qquad \Jo\eqdef I_{dN}-J \ ,
\end{equation}
resp.  the projector onto the \emph{consensus subspace} $\left\{ \un\otimes
  \theta : \theta\in\bR^d\right\}$ and the projector onto the orthogonal
subspace.  For any vector ${\bs x} \in\bR^{dN}$, define the vector of $\bR^d$
\begin{equation}
  \label{eq:average}
  \la \bs x \ra \eqdef \frac 1N ({\un^T\otimes I_d})\bs x  \ . 
\end{equation}
Note that $\la \bs x \ra=(x_1+\dots+x_N)/N$ in case we write $\bs x
=(x_1^T,\dots,x_N^T)^T$ for some $x_1,\dots, x_N$ in $\bR^d$. We denote by $\bs x_\bot := \Jo\bs x$
the projection of $\bs x$ on the orthogonal to the consensus space. Remark that
$\bs x=\un\otimes \la \bs x \ra+\bs x_\bot$. In particular, we set $\othn := \Jo\thn$ where $\Jo$
is given by (\ref{eq:projectors}) and refer to $\othn$ as the \emph{disagreement vector}.
Denote by: $$f:=\frac 1N\sum_{i=1}^N
f_i$$ the average of utility functions. Define the set of KKT points of $f$ on
$G$ (also called the set of stationary points) as:
\begin{equation}
\label{eq:KKT}
\cL := \left\{ \theta \in G\ : \ -\nabla f(\theta) \in \cN_G(\theta)\right\}\ ,
\end{equation}
where $\cN_G(\theta)$ is the normal cone \emph{i.e.}, $\cN_G(\theta):=\{ v \in
\bR^d\ :\ \forall \theta'\in G, v^T (\theta-\theta')\geq 0\}$. 
Define
$\sd(\bs x,A):=\inf\{ |\bs x-a|\,:a \in A\}$ for any $\bth\in\bR^{dN}$ and any set
$A$. We say that a random sequence $(\bs x_n)_{n\geq 1}$ converges almost surely (a.s.) to a set $A$ if 
$\sd(\bs x_n,A)$ converges a.s. to zero as $n\to\infty$. Please note that convergence to $A$ does not imply convergence to some point of $A$.

\subsection{Main result}

\begin{theo} 
  \label{the:cvc}
  Assume that $f(\cL)$ has an empty interior.
  Under Assumptions~\ref{hyp:model},~\ref{hyp:step},~\ref{hyp:KT}, the sequence $(\thn)_{n\geq 1}$ converges a.s. to the set
  $\{\un\otimes\theta\ :\ \theta\in\cL\}$.
  Moreover, $(\athn)_{n\geq 1}$ converges a.s. to a connected component of $\cL$.
\end{theo}

Theorem~\ref{the:cvc} establishes two points. First, a consensus is achieved as
$n$ tends to infinity, meaning that $\max_{i,j} |\theta_{n,i}-\theta_{n,j}|$
converges a.s. to zero. Second, the average estimate $\athn$ converges to the
set $\cL$ of KKT points.  As a consequence, if $\cL$ contains only isolated
points, sequence $\athn$ converges a.s. to one of these points.

In particular, when $f$ is convex, $\athn$ converges to the set of global
solutions to the minimization problem~(\ref{eq:pb}).  However, as already
remarked, our result is more general and does not rely on the convexity of
$f$. If $f$ is not convex, sequence $\athn$ does not necessarily converge to a
global solution. Nevertheless, it is well known that the KKT conditions are
satisfied by any local minimizer \cite{borwein:2006}.

The condition that $f(\cL)$ has an empty interior is satisfied in most practical cases.
From Sard's theorem, it holds as soon as $f$ is $d$ times continuously differentiable.

\section{Proof of Theorem~1}
\label{sec:proofcvc}

\subsection{Sketch of Proof}

We first provide some insights on the proof (all statements are made rigorous in 
the next subsections).

The proof is decomposed in the following steps. 

\begin{enumerate}
\item First, establish convergence to consensus with probability 1 (see
  Lemma~\ref{lem:agreement}).  In this step we show that iterating matrices
  $W_n$ yields to consensus whatever the behavior of sequence $(\thn)_{n\geq 1}$. See
  subsection~\ref{subsec:consensus}.

\item Show then that $\athn$ is ruled by the following discrete time dynamical
  system (see Proposition~\ref{prop:average} in subsection~\ref{subsec:avdyn}):
  \begin{multline}
    \label{eq:jojo}
  \athn = \athnmu - \gamma_n \nabla f(\athnmu) + \gamma_n g_{\gamma_n}(\thnmu)  \\
  +\gamma_n\xi_n+\gamma_n r_n
  \end{multline}
In order to give some insight, assume just for a moment that $g$ is identically $0$. 
In that case, one could write: 
  \[
  \frac{\athn - \athnmu}{\gamma_n} = -\nabla f(\athnmu) + \xi_n + r_n\,,
  \]
  and view $\athn$ as a noisy discrete approximation of the well studied
  Ordinary Differential Equation (ODE):
  \[
  \dot x = -\nabla f(x)\;,
  \]
  where $\dot x$ stands for the derivative of function $t\mapsto x(t)$.
  See, for instance, \cite{hirsh:smale:1974}.
  
  This line of reasoning is the so-called ``ODE'' method (see, for instance,
  \cite{borkar:2008}). If function $g$ was regular enough, one could still use the
  ODE method and study an ODE of the form $\dot x = -\nabla f(x) + h(x)$
  where $h$ is a function derived from $g$ (we skip the details, which are not
  important here).  Unfortunately, in our case, function $g$ is not regular enough. 
  In that case, the standard ODE method fails for two reasons:
  {\sl i)} such an ODE might have several solutions, {\sl ii)} it does no longer ensure that
    discretized versions of the ODE stay close to suitable solutions of
    the original ODE.


\item The framework of Differential Inclusions (DI) addresses these two
  issues. The ODE is replaced by the DI
  \[
  \dot x\in F(x)\,,
  \]
  (see~(\ref{eq:di}) in what follows), where, at any time instant $t$,  $F(x(t))$ is a \emph{set} of
  vectors instead of the single vector as in an ODE.
An important property is that the set-valued function $F$ should now be
upper semicontinuous.
It remains to prove two assertions to finish the proof:
  \begin{enumerate}
  \item Show that the noisy discretized dynamical system (\ref{eq:jojo})
asymptotically behaves ``the same way'' as a solution to the DI.
To formalize this, we shall rely on the notion of \emph{perturbed differential inclusions}
of~\cite{benaim:2005}. We prove in Section~\ref{subsec:interp} that
the continuous-time process obtained from a suitable interpolation of (\ref{eq:jojo})
is in fact a perturbed solution to the DI. Using the results of~\cite{benaim:2005}
which we recall in Section~\ref{subsec:di}, the limit set of the interpolated process
can be characterized by simply studying the behavior of the solutions $x(t)$ to the DI for large $t$.

  \item The asymptotic behavior of the DI is addressed in Section
  \ref{subsec:didynamics} where it is shown that $f$ is a Lyapunov function for
  the set $\cL$ of KKT points under the dynamics induced by the DI.
  \end{enumerate}
\end{enumerate}


\subsection{Preliminaries: Useful Facts about Set-Valued Dynamical Systems}
\label{subsec:di}

Before providing the details of the proof, we recall some useful facts about
perturbed differential inclusions.  All definitions and statements made in this
paragraph can be found in reference~\cite{benaim:2005}.  However, for the sake of
readability and completeness, it is worth recalling some facts.

Consider an arbitrary set-valued function $F$ which maps each point
$\theta\in\bR^d$ to a set $F(\theta)\subset\bR^d$.  Assume that $F$ satisfies
the following conditions:
\begin{cond}
  \label{cond:usc}
  The following holds:
  \begin{itemize}
  \item $F$ is a closed set-valued map \emph{i.e.}, $\{(\theta,y):y\in
    F(\theta)\}$ is a closed subset of $\bR^d\times \bR^d$.
  \item For any $\theta\in\bR^d$, $F(\theta)$ is a nonempty compact convex
    subset.
  \item There exists $c>0$ such that for any $\theta\in\bR^d$, $\sup_{z\in
      F(\theta)}|z|<c(1+|\theta| )$.
  \end{itemize}
\end{cond}
A function $x:\bR\to\bR^d$ is called a solution to the differential inclusion
\begin{equation}
  \label{eq:di}
  \frac{dx}{dt}\in F(x)
\end{equation}
if it is absolutely continuous and if $\frac{dx(t)}{dt}\in F(x(t))$ for
almost every $t\in\bR$. 
Let $\Lambda$ be a compact set in $\bR^d$.
A continuous function $V:\bR^d\to\bR$ is called a \emph{Lyapunov function} for $\Lambda$ if, 
for any solution $x$ to~(\ref{eq:di}) and for any $t>0$, 
$$
V(x(t)) \leq V(x(0))
$$
and where the inequality is strict whenever $x(0)\notin\Lambda$.

Finally, a function $y:[0,\infty)\to\bR^d$ is called a \emph{perturbed
  solution} to~(\ref{eq:di}) if it is absolutely continuous and if there exists
a locally integrable function $t\mapsto U(t)$ such that:
\begin{itemize}
\item For any $T>0$, $\lim_{t\to\infty}\sup_{0\leq v\leq T}\left|
    \int_t^{t+v}U(s)ds\right| =0$,
\item There exists a function $\delta:[0,\infty)\to[0,\infty)$ such that
  $\lim_{t\to\infty}\delta(t)=0$ and such that for almost every $t>0$,
  $\frac{dy(t)}{dt}-U(t)\in F^{\delta(t)}(y(t))$, where for any
  $\delta>0$:
\end{itemize}  \begin{equation}
    \label{eq:Fdelta}
   F^\delta(\theta) := \left\{ z\in \bR^d\,:\, \exists \theta', |\theta-\theta'|<\delta , \sd(z,F(\theta'))<\delta\right\}.
  \end{equation}

\begin{remark}
  Note that there is no guarantee that a perturbed solution remains close to a
  solution from a given time (for instance, $y(t) = \log(1+t)$ is a perturbed
  solution to $\dot x =0$). However, loosely speaking, $y$ is a perturbed
  solution if the behavior of $y$ on the interval $[t,t+T]$ shadows the
  differential inclusion for $t$ large enough, for each \emph{fixed} width $T$.
\end{remark}

The following result, found in reference~\cite{benaim:2005}, will be used in our
proofs. Denote by $\overline{\cS}$ the closure of a set $\cS$.
\begin{theo}[\cite{benaim:2005}]
  Let $V$ be a Lyapunov function for $\Lambda$. Assume that $V(\Lambda)$ has an
  empty interior.  Let $y$ be a bounded pertubed solution to~(\ref{eq:di}). Then,
$$
\bigcap_{t\geq 0} \overline{y([t,\infty))} \subset \Lambda\ .
$$
\label{the:benaim}
\end{theo}
\begin{proof}
  The result is a consequence of Theorem 4.2, Theorem 4.3 and Proposition 3.27
  in~\cite{benaim:2005}.
\end{proof}

\subsection{Agreement between Agents}
\label{subsec:consensus}


\begin{lemma}[Agreement]
\label{lem:agreement}
Assume that $G$ is a compact convex set.  Under Assumptions~\ref{hyp:model}
and~\ref{hyp:step}, $\sum_{n\geq 1}\bE\left|\othn\right|^2<\infty$. 
As a consequence, $\othn$ converges to zero a.s..
\end{lemma}

\begin{proof}
  We rewrite~(\ref{eq:algo}) as $\thn = (W_n\otimes I_d)(\thnmu+\gamma_n
  \bsZ_n)$ where
  \begin{equation}
    \bsZ_n := \frac{P_{G^N}\left[\thnmu+\gamma_{n}\bY_n\right]-\thnmu}{\gamma_n}\ .
    \label{eq:Zn}
  \end{equation}
  Before going into the details of the proof of Lemma~\ref{lem:agreement}, it is
  worth noting that $|\bsZ_n|^2 \leq |\bY_n|^2$ (just note that
  $\thnmu=P_{G^N}[\thnmu]$ and use the fact that $G^N$ is convex).  By
  Assumption~\ref{hyp:model}(f), the sequence $(\bE[|\bsZ_n|^2])_{n\geq 1}$ is
  therefore bounded.

  We now study $\othn$. As $W_n\un=\un$, it is straightforward to show that
  $\Jo(W_n\otimes I_d)=\Jo(W_n\otimes I_d)\Jo$.  As a consequence, $\othn =
  \Jo(W_n\otimes I_d)(\othnmu+\gamma_n \bsZ_n)$. Using the so-called ``mixed
  product rule'', ($ [A\otimes B]\cdot[C\otimes D] = AC\otimes BD $), we
  expand the square Euclidean norm of the latter vector:
$$
|\othn|^2 = (\othnmu+\gamma_n \bsZ_n)^T\cW_n(\othnmu+\gamma_n \bsZ_n)\ .
$$
where $\cW_n :=(W_n^T\otimes I_d)\left((I_N-\un\un^T/N)\otimes I_d\right)^2(W_n\otimes I_d)$.
Note that $\cW_n = \{W_n^T(I_N-\un\un^T/N)W_n\}\otimes
I_d$.
Integrate both sides of the above equation w.r.t. the random variable $W_n$:
$$
\bE[|\othn|^2\,|\cF_{n-1},\bsZ_n] \leq \rho_n |\othnmu+\gamma_n \bsZ_n|^2\ .
$$
Expand the righthand side and take the expectation. Using the fact that $\rho_n<1$ for
$n$ large enough,
\begin{multline*}
  \bE[|\othn|^2] \leq \rho_n \bE[|\othnmu|^2]+2\gamma_n\bE[|\othnmu|\,
|\bsZ_n|] \\+\gamma_n^2\bE[|\bsZ_n|^2]\ .
\end{multline*}
As $\bE[|\bsZ_n|^2]$ is uniformly bounded, we obtain from Cauchy-Schwartz's
inequality:
$$
\bE[|\othn|^2] \leq \rho_n
\bE[|\othnmu|^2+\gamma_n\sqrt{C\,\bE[|\othnmu|^2]}+\gamma_n^2C
$$
for some constant $C>0$.

Let us denote $v_n:=\bE[|J^\perp\bth_n|^2]$. Using $\tilde\gamma_n=\sqrt
C\gamma_n$, which also fulfills Assumption~\ref{hyp:step}, we get:
\begin{eqnarray}
  v_{n}&\leq& \rho_n v_{n-1} + \tilde\gamma_n\sqrt{v_{n-1}} + \tilde\gamma_n^2\ .
\end{eqnarray}
Let $u_n:=n^{2\alpha}v_n$ for some $\alpha>1/2$ satisfying~(\ref{eq:condgamma})
and~(\ref{eq:condrho}). Then,
\begin{multline}
  u_n \leq \left(1+\frac1{n-1}\right)^{2\alpha}\rho_n u_{n-1} \\+ 
  n^\alpha\tilde\gamma_n\left(1+\frac1{n-1}\right)^{\alpha}\sqrt{u_{n-1}} + n^{2\alpha}\tilde\gamma_n^2\ .
\end{multline}
This implies in turn:
\[
u_n - u_{n-1} \leq \left(-a_nu_{n-1} +
  b_n\sqrt{u_{n-1}}+c_n\right)n^\alpha\tilde\gamma_n\,,
\]
where $b_n = (1+\frac 1{n-1})^{2\alpha}$, $a_n =
\frac{1-\sqrt{b_n}\rho_n}{n^\alpha\tilde\gamma_n}$, and $c_n =
n^\alpha\tilde\gamma_n$. A straightforward analysis of function $\phi_n:u\mapsto -a_nu
+ b_n\sqrt u + c_n$ shows that $u>t_n$ implies $\phi_n(u) < 0$ where $t_n:=
(b_n/a_n+c_n/b_n)^2$. Note that, using Assumption~\ref{hyp:model}(b), $a_n\sim
\frac{1-\rho_n}{n^\alpha\tilde\gamma_n}$ and using Assumption~\ref{hyp:step}, $t_n$ is
bounded above, say by a constant $K>0$. Moreover, when $u\leq t_n$,
$\phi_n(u)\leq \phi_n(b_n/2a_n) = c_n + {b_n}^2/4{a_n}$. Notice again that
$\phi_n(b_n/2a_n)$ is bounded above, say by a constant $L>0$. We have proved
that if $u_{n-1}\leq K$ then $u_n\leq K+L$ and if $u_{n-1}> K$, $u_n\leq
u_{n-1}$. This implies that $u_n\leq\max(K+L,u_0)$. Hence $\sum v_n<\infty$.
\end{proof}

Lemma~\ref{lem:agreement} proves that agents asymptotically reach an agreement
on their estimate. Another way to state Lemma~\ref{lem:agreement} is to write
that $\max_{i,j=1\cdots N} |\theta_{n,i}-\theta_{n,j}|$ converges a.s. to zero
as $n$ tends to infinity.  Therefore, the asymptotic analysis of the whole
vector $\bs\theta_n$ now reduces to the study of the average $\athn =
N^{-1}\sum_{i=1}^N\theta_{n,i}$.

\subsection{Average Estimate}
\label{subsec:avdyn}


For any $\gamma>0$, $\bth\in G^N$, define $g_\gamma(\bth) := \gamma^{-1}(\frac{\un^T}N\otimes I_d)\int\left(P_{G^N}[\bth+\gamma
    \by]-\bth-\gamma \by\right)\sd\mu_{\bth}(\by)$.
Under Assumption~\ref{hyp:model}(c), this means that:
$$
g_{\gamma_n}(\thnmu) \!=\! 
\frac{\la \bE \left[\!P_{G^N}[\thnmu\!+\!\gamma_n\bY_n]\!-\!\thnmu\!-\!\gamma_n \bY_n\big|\cF\nmu\right]\ra}{\gamma_n} .
$$
\begin{prop}
  \label{prop:average}
  Under Assumptions~\ref{hyp:model},~\ref{hyp:step},~\ref{hyp:KT}, there exists
  two stochastic processes $(\xi_n)_{n\geq 1}$, $(r_n)_{n\geq 1}$ such that for any $n\geq 1$:
  \begin{equation}
    \athn = \athnmu - \gamma_n \nabla f(\athnmu) + \gamma_n g_{\gamma_n}(\thnmu) +\gamma_n\xi_n+\gamma_n r_n
    \label{eq:rmc}
  \end{equation}
  and satisfying w.p.1:
  \begin{gather}
    \lim_{n\to\infty}\sup_{k\geq n}\left|\sum_{\ell=n}^k \gamma_\ell\, \xi_\ell\right|=0 \label{eq:supxi}\\
    \lim_{n\to\infty}r_n=0\ . \nonumber
  \end{gather}
\end{prop}
The proof is provided in Appendix~\ref{app:average}.

Note that the third term in the righthand side of~(\ref{eq:rmc}) is zero
whenever $\thnmu+\gamma_n\bY_n$ lies in $G^N$ \emph{i.e.}, when the projector is
inoperant. In order to have some insights, assume just for a moment that this
holds for any $n$ after a certain value. In this case, equation~(\ref{eq:rmc})
simply becomes
\begin{equation}
  \athn = \athnmu - \gamma_n \nabla f(\athnmu)+\gamma_n\xi_n+\gamma_n r_n\ .
  \label{eq:athnsimpl}
\end{equation}
In this case, by the continuity of $\nabla f$ and using the above conditions on
the sequences $\xi_n$ and $r_n$, the asymptotic behavior of sequence $\athn$ can
be directly characterized using classical stochastic approximation results
(see~\cite{kushner:2003,andrieu:2005,benaim:1999,delyon:2000}).  Indeed, a
sequence $\athn$ satisfying~(\ref{eq:athnsimpl}) converges to the set of
critical points of $f$.  Nevertheless, the projector $P_{G^N}$ is generally
active in practice, so that the term $g_{\gamma_n}(\thnmu)$ may be nonzero
infinitely often.  This additional term raises at least two problems. First, it
depends on the whole vector $\thn$ and not only on the average $\athn$:
equation~(\ref{eq:rmc}) looks thus nothing like a usual iteration of a
stochastic approximation algorithm. Second, $g_\gamma(\bth)$ is not a continous
function of $\bth$, whereas standard approaches often assume the continuity of
the mean field of the stochastic approximation algorithm.

\subsection{Interpolated Process}
\label{subsec:interp}

Define $\mu:= \sup_{\bth\in G^N}\int |\by|\sd\mu_{\bth}(\by)$. Define the following
set-valued function $F$ on $\bR^d$ which maps any  $\theta$ to the
set:
\begin{equation}
  F(\theta) := \left\{ -\nabla f(\theta)-z  : z\in\cN_G(\theta),\, |z|\leq 3\mu\right\}\ . \label{eq:F}
\end{equation}
Using the fact that $f$ is continuously differentiable and that $G$ is closed
and convex, it can be shown that $F$ satisfies Condition~1\footnote{As a purely
  technical point, note that the third point in Condition~1 is satified only if
  $|\nabla f(\theta)|$ increases at most at linear speed when
  $|\theta|\to\infty$, which has of course no reason to be true in general. This
  is however unimportant, as the values of $\theta$ will always be restricted to
  the bounded set $G$ in the sequel.  Moreover, for $\theta\notin G$, one can
  always redefine $F(\theta)$ as in~(\ref{eq:F}) but replacing $f$ with $\varphi
  \circ f$ where $\varphi$ is a slowly increasing map chosen such that
  Condition~1 holds for $\varphi \circ f$. }.
Recall notation
$F^\delta(\theta)$ in~(\ref{eq:Fdelta}).  Consider stochastic processes
$(\xi_n,r_n)_{n\geq 1}$ as in Proposition~\ref{prop:average}.
\begin{prop}
  \label{prop:gF}
Under Assumptions~\ref{hyp:model}, \ref{hyp:step}, \ref{hyp:KT}, 
there exists a sequence of random variables $(\delta_n)_{n\geq 1}$ 
converging a.s. to zero and an integer $n_0$ such that for any $n\geq n_0$,
$$
\frac{\athn - \athnmu}{\gamma_n} -\xi_n- r_n \in F^{\delta_n}(\athnmu)\ .
$$
\end{prop}

\begin{proof}
  From the definition of $g_\gamma$, it is straightforward to show that
  $|g_\gamma(\bth)|\leq 2\mu$.  Note that for any $\theta\in G$ and any $y\in
  \bR^d$, the vector $\theta+\gamma y-P_{G}[\theta+\gamma y]$ belongs to the
  normal cone $\cN_G(P_{G}[\theta+\gamma y])$ at point $P_{G}[\theta+\gamma y]$.
  Otherwise stated, $P_{G}[\theta+\gamma y]-\theta-\gamma y$ can be written as a
  linear combination of the gradient vectors associated with the active
  constraints, where the coefficients of the linear combination are nonnegative
  (see for instance, \cite{rockafellar:wets:1998}, theorem 6.14). The latter
  linear combination is moreover unique due to the qualification constraint
  given by Assumption~\ref{hyp:KT}(c). More precisely, if $A(P_{G}[\theta+\gamma
  y])$ represents the active set at point $P_{G}[\theta+\gamma y]$ for any
  $\theta$, $\gamma$, $y$, there exists a unique collection of nonnegative
  coefficients $(\lambda_j(\theta,\gamma,y) : j\in A(P_{G}[\theta+\gamma y]))$
  such that:
\begin{multline}
  -\frac{P_{G}[\theta+\gamma y]-\theta-\gamma y}{\gamma} =\\ \sum_{j\in A(P_{G}[\theta+\gamma y])} \!\!\!\!\lambda_j(\theta,\gamma,y) \nabla q_j(P_{G}[\theta+\gamma y])\ .
  \label{eq:decomp_cone}
\end{multline}
Throughout the paper, we use the convention that $\lambda_j(\theta,\gamma,y)=0$
in case $j\notin A(P_{G}[\theta+\gamma y])$.  The following technical lemma is
proved in Appendix~\ref{app:lambda}.
\begin{lemma}
  \label{lem:lambda}
  Under Assumptions~\ref{hyp:model}(f) and~\ref{hyp:KT},
  \begin{equation}
    \sup_{\substack{\bth\in G^N,\,\gamma>0\\ i=1\cdots N,j=1\cdots p}} \int \lambda_j(\theta_i,\gamma,y_i)^2\sd\mu_\bth(y_1,\dots,y_N) <\infty\ .
    \label{eq:suplambda}
  \end{equation}
\end{lemma}
We rewrite $g_\gamma(\bth)$ using expansion~(\ref{eq:decomp_cone}) as:
\begin{multline*}
 \!\!\!\! - g_\gamma(\bth) = \frac 1{N}\sum_{i=1}^N\ \int\!\!\!\! \sum_{j\in
  A(P_{G}[\theta_i+\gamma y_i])} \!\!\!\!\!\!\!\!\!\!\lambda_j(\theta_i,\gamma,y_i)
\nabla q_j(P_{G}[\theta_i+\gamma y_i])\\ \times \sd\mu_\bth(y_1,\dots,y_N)\ .
\end{multline*}
The following function $\phi:[0,+\infty)\to[0,+\infty)$ will be useful. Define:
\begin{equation}
  \phi(x):=\sup_{\substack{(\theta,\theta')\in G^2 : |\theta-\theta'|\leq x \\ j=1\cdots p}} \left|\nabla q_j(\theta)-\nabla q_j(\theta')\right|\ .
  \label{eq:phi}
\end{equation}
Since each gradient $\nabla q_j$ is continuous, it is uniformly continuous on
the compact set $G$. Thus $\phi(x)$ tends to zero as $x\downarrow 0$.  Loosely
speaking, when $\gamma$ is small and when all $\theta_i$'s are close to the
average $\ath$, the point $P_{G}[\theta_i+\gamma y_i]$ is close to $\ath$. In
this case, the uniform continuity of $\nabla q_j$ implies that $\nabla
q_j(P_{G}[\theta_i+\gamma y_i])\simeq \nabla
q_j(\ath)$. Lemma~\ref{lem:Aepsilon} below states a somewhat stronger result.
For any $\epsilon\geq 0$ and $\theta\in G$, define $\cA(\theta,\epsilon)$ as the
set of constraints which are active at least for some point in an
$\epsilon$-neighborhood of $\theta$:
\begin{equation}
  \cA(\theta,\epsilon) := \{ j=1,\cdots,p\,:\, \sd(\theta,q_j^{-1}(\{0\}))\leq \epsilon\}\ .
  \label{eq:Aepsilon}
\end{equation}
\begin{lemma}
\label{lem:Aepsilon}
Under Assumption \ref{hyp:KT}, there exists a constant $C>0$ and a function $\gamma\mapsto \epsilon(\gamma)$
on $[0,+\infty)$ satisfying $\lim_{\gamma\downarrow 0}\epsilon(\gamma)=0$ such that the following
holds.  For any $\bth\in G^N$ and any $\gamma>0$, there exists $(\alpha_1,\cdots,\alpha_p)\in
[0,C]^p$ such that
\begin{equation}
\label{eq:approxg}
\left|-g_\gamma(\bth) - \!\!\!\!\!\!\sum_{j\in \cA(\ath,\epsilon(\gamma\lor|\oth|))} \!\!\!\!\!\!\!\!\!\!\!\alpha_{j}\,\nabla q_j(\ath)\right|\leq \epsilon(\gamma\lor|\oth|)\ .
\end{equation}
\end{lemma}
The proof is provided in Appendix~\ref{app:Aepsilon}. The sum in the lefthand
side of~(\ref{eq:approxg}) is a (nonnegative) linear combination of the gradient
vectors of the constraints at point $\ath$.  However, this does not necessarily
imply that this term belongs to the normal cone $\cN_G(\ath)$ because, for a
fixed $\epsilon>0$, the set $\cA(\ath,\epsilon)$ is in general larger than the
active set $A(\ath)$.  Nevertheless, the following lemma states that
$\cA(\ath,\epsilon)$ is no larger than a certain active set $A(\theta')$ for
some $\theta'$ in a neighborhood~of~$\ath$.
\begin{lemma}
\label{lem:AA}
Under Assumption \ref{hyp:KT}, there exists a function $\epsilon\mapsto
\delta(\epsilon)$ on $[0,+\infty)$ satisfying $\lim_{\epsilon\downarrow
  0}\delta(\epsilon)=0$ and there exists $\epsilon_0>0$ such that for any
$0<\epsilon<\epsilon_0$ and any $\theta\in G$, there exists $\theta'\in G$ s.t.:
\begin{equation}
\label{eq:AA}
|\theta-\theta'|<\delta(\epsilon)\quad\textrm{ and }\quad \cA(\theta,\epsilon)\subset A(\theta')\ .
\end{equation}
\end{lemma}
The proof is given in Appendix~\ref{app:AA}.  We put all pieces together.
Consider constant $C$ and functions $\epsilon(\,.\,)$ and $\delta(\,.\,)$ as in
Lemma~\ref{lem:Aepsilon} and~\ref{lem:AA} respectively. Define $\epsilon_n :=
\epsilon(\gamma_n\lor|\othnmu|)$ and $\delta_n := \max(
\epsilon_n+Cp\,\phi(\epsilon_n)\,,\, \delta(\epsilon_n))$.  Clearly,
$\epsilon_n$ (and consequently $\delta_n$) converges to zero a.s. due to
Lemma~\ref{lem:agreement} and to the fact that $\gamma_n\to 0$. In particular,
there exists an integer $n_0$ s.t.  $\epsilon_n<\epsilon_0$ for any $n\geq n_0$.
By Lemma~\ref{lem:AA}, for any $n\geq n_0$, there exists $\theta_n'\in G$
satisfying $|\theta_n'-\athnmu|<\delta(\epsilon_n)$ and
$\cA(\athnmu,\epsilon_n)\subset A(\theta_n')$. Thus, by
Lemma~\ref{lem:Aepsilon}, there exists $(\alpha_1,\cdots,\alpha_p)\in [0,C]^p$,
such that
\begin{equation}
  |-g_{\gamma_n}(\thnmu)-\sum_{j\in A(\theta_n')} \!\!\alpha_{j}\,\nabla q_j(\athnmu)| \leq \epsilon_n\ .
  \label{eq:gn}
\end{equation}
Define $z_n := \sum_{j\in A(\theta_n')} \alpha_{j}\,\nabla q_j(\theta_n')$.
Clearly, $z_n \in \cN_G(\theta_n')$. Using inequality~(\ref{eq:gn}),
\begin{eqnarray*}
  \left|-g_{\gamma_n}(\thnmu)-z_n\right| \!\!\!\!
  &\leq&\!\!\!\! \epsilon_n+C\!\!\!\! \sum_{j\in A(\theta_n')}\!\!\!\!\left|\nabla q_j(\theta_n')\!-\!\nabla q_j(\athnmu)\right| \\
  &\leq&\!\!\!\! \epsilon_n+Cp\, \phi(|\theta_n'-\athnmu|) \quad \leq \delta_n\ .
\end{eqnarray*}
As $|g_\gamma(\theta)|\leq 2\mu$, this moreover implies that
$|z_n|\leq 3\mu$ provided that $\delta_n$ is small enough. Thus,
$$
\sd(-g_{\gamma_n}(\thnmu)\,,\,\cN_G(\theta_n')\cap \{ z:|z|\leq 3\mu\})\leq
\delta_n
$$
for all but a finite number of $n$'s.  The proof of Proposition~\ref{prop:gF} is
completed by using~(\ref{eq:rmc}).
\end{proof}

Define $\tau_0=0$ and $\tau_n:=\sum_{i=k}^n\gamma_k$ for any $n\geq 1$. Define
the continuous time process $\Theta:[0,+\infty)\to \bR^d$ as:
$$
\Theta(\tau\nmu+t) := \athnmu + t\,\frac{\athn-\athnmu}{\tau_n-\tau\nmu}
$$
for any $t\in [0,\gamma_n)$ and any $n\geq 1$.
\begin{prop}
  \label{prop:di}
  Under Assumptions~\ref{hyp:model},~\ref{hyp:step},~\ref{hyp:KT}, the
  interpolated process $\Theta$ is a perturbed solution to~(\ref{eq:di}) w.p.1.
\end{prop}

\begin{proof}
  The proof follows more or less the same idea as the proof of Proposition~1.3
  in reference~\cite{benaim:2005}.  There exists an event $\Omega_0$ of probability one
  such that $\delta_n\to 0$ and $r_n\to 0$ for any sample point
  $\omega\in\Omega_0$. From now on, we fix such an $\omega$ and we study
  function $\Theta$ for this fixed sample point.  For any $n\geq 1$ and
  $\tau_{n-1}<t<\tau_{n}$, $\frac{d\Theta(t)}{dt}=(\athn-\athnmu)/\gamma_n$.  By
  Proposition~\ref{prop:gF},
$$
\frac{d\Theta(t)}{dt}\ \in\ \xi_n + r_n + F^{\delta_n}(\athnmu)\qquad (
\tau_{n-1}<t<\tau_{n})\ .
$$
The following property is easy to check. For any set-valued function $F$, any
$r\in\bR^d$, $\delta>0$,
$$
\forall (\theta,\theta')\in\bR^d\times \bR^d,\ r+F^\delta(\theta) \subset
F^{\delta+|r|+|\theta-\theta'|}(\theta')\ .
$$
Now, for any $n$ and any $\tau_{n-1}<t<\tau_{n}$, define
$\eta(t):=\delta_n+|r_n|+|\Theta(t)-\athnmu|$ and $U(t)=\xi_n$.  We obtain:
$$
\frac{d\Theta(t)}{dt}-U(t)\ \in\ F^{\eta(t)}(\Theta(t))\ .
$$
It is straightforward to show from~(\ref{eq:supxi}) that for any $T>0$,
$\sup_{0\leq v\leq T}\left| \int_t^{t+v}U(s)ds\right|$ tends to zero as
$t\to\infty$ (we refer the reader to Proposition~1.3 in reference~\cite{benaim:2005} for
details). We now prove that $\eta(t)$ tends to zero as $t\to\infty$.  To this
end, note that for any $\tau_{n-1}<t<\tau_{n}$,
$$
|\eta(t)| \leq  \delta_n+|r_n|+|\athn-\athnmu| \ ,
$$
thus,
\begin{multline*}
  |\eta(t)| \leq \delta_n+(1+\gamma_n)|r_n|+\gamma_n \sup_{\theta\in G}\nabla f(\theta) 
+ \gamma_n |g_{\gamma_n}(\thnmu)| \\+ \gamma_n |\xi_n|\ .
\end{multline*}
The first three terms of the righthand side of the above inequality converge to
zero as $\gamma_n\to0$, $r_n\to 0$ and $\delta_n\to 0$.  The fourth term tends to
zero as well because $|g_{\gamma}(\bs\theta)|$ is uniformly bounded in
$(\gamma,\theta)$, as remarked in the proof of
Proposition~\ref{prop:gF}. Finally, $\gamma_n\xi_n$ tends to zero
by~(\ref{eq:supxi}).  Thus $\eta(t)$ tends to zero as $t\to\infty$. This
completes the proof of Proposition~\ref{prop:di}.
\end{proof}

\subsection{Study of the Differential Inclusion Dynamics}
\label{subsec:didynamics}

\begin{prop}{\cite{aubin:cellina:1984}}
  \label{prop:aubin}
  Any solution $x$ of the differential inequality $\frac{d x}{dt}\in F(x)$ with
  $F$ defined by eq.~(\ref{eq:F}) such that $\forall t\in\bR$, $x(t)\in G$
  satisfies for almost all $t\in\bR$:
  \begin{equation}
    \label{eq:projode}
  \frac{d x}{dt}(t) = P_{\cT_G(x(t))}(-\nabla f(x(t)))
  \end{equation}
  where $P_{\cT_G(x)}$ stands for the projection onto the tangent cone
  $\cT_G(x)$ at point $x$.
\end{prop}
\begin{proof}
  For the sake of completeness, we reproduce here the proof of reference
  \cite[pp. 266]{aubin:cellina:1984}. Let $x(t)$ be a solution of
  eq.~(\ref{eq:di}); it is differentiable for almost every $t$ by definition. At
  such $t$ one has $\frac{dx}{dt}(t) = \lim_{\epsilon\downarrow
    0}\frac{x(t+\epsilon)-x(t)}{\epsilon}\in \cT_G(x(t))$ since $x(t)\in G$ for
  all $t$. For the same reason (using $\epsilon<0$), one has
  $-\frac{dx}{dt}(t)\in\cT_G(x(t))$. By convexity of $G$, $\cT_G(\theta)$ is the
  dual cone of $\cN_G(\theta)$ for every $\theta$. Hence $\forall
  v\in\cN_G(x(t))$ one has at the same time $\langle \frac{dx}{dt}(t),v\rangle
  \leq 0$ and $\langle -\frac{dx}{dt}(t),v\rangle \leq 0$. As a consequence,
  $\forall v\in\cN_G(x(t))$,$\langle \frac{dx}{dt}(t),v\rangle =0$. Now,
  considering that $x$ is a solution of eq.~(\ref{eq:di}) with $F$ defined by
  eq.~(\ref{eq:F}): $\frac{dx}{dt}(t) = -v - \nabla f(x(t))$ for some
  $v\in\cN_G(x(t))$. To conclude, $-\nabla f(x(t))$ can be written
  $v+\frac{dx}{dt}(t)$ with $\langle \frac{dx}{dt}(t),v\rangle = 0$ and it is a
  classical fact from convex analysis (see, for instance,
  reference~\cite{rockafellar:wets:1998}) that, $\cN_G(x(t))$ and $\cT_G(x(t))$
  are dual cones: $v = P_{T}\left(-\nabla f(x(t))\right)$.
\end{proof}

The following proposition is straightforward to prove but has an important role.
\begin{prop}
  \label{prop:lyap}
  Any solution $x$ of eq.~(\ref{eq:projode}) admits $f$ as a Lyapunov function for the set $\cL$ of KKT points defined by eq.~(\ref{eq:KKT}).
\end{prop}
\begin{proof}
  One has $\frac{d}{dt}f(x(t)) = \langle \nabla
  f(x(t)),\frac{dx}{dt}(t)\rangle$. From Proposition~\ref{prop:aubin}, we deduce:
  \[
  \frac{d}{dt}f(x(t)) = - \langle -\nabla f(x(t)), P_{\cT_G(x(t))}(-\nabla
  f(x(t)))\rangle
  \]
  When $T$ and $N$ are dual cones, the decomposition $v = P_{T}(v) + P_{N}(v)$, $\langle
  P_T(v),P_N(v)\rangle = 0$ holds. Using this decomposition with $v=-\nabla f(x(t))$,
  $T = \cT_G(x(t))$ and $N = \cN_G(x(t))$, one deduces
  \[
  \frac{d}{dt}f(x(t)) = - | P_{\cT_G(x(t))}\left(-\nabla f(x(t))\right)|^2\;.
  \]
  This gives the sought result.
\end{proof}

We are now in a position to apply Theorem~\ref{the:benaim} with Lyapunov
function $f$ itself and $\Lambda$ the set of KKT points of our optimization
program (see Proposition~\ref{prop:lyap}).


\section{Application: Power Allocation in Ad-hoc Wireless Networks}
\label{sec:appli}


The context of power allocation for wireless networks has recently raised a
great deal of attention in the field of distributed optimization and game
theory~\cite{scutari:2010,belgama:2010,samson:2011}. Application of distributed
optimization to power allocation has been previously investigated
in~\cite{ram:veeravalli:nedic:INFOCOM:2009}. 

Consider an \emph{ad-hoc} network composed of $N$ source-destination pairs.  
We focus on the so-called interference channel \emph{i.e.},
the signal received by the destination of a given pair $i=1,\dots,N$ is corrupted both by 
an additive white Gaussian noise of variance $\sigma_i^2$ and by the interference produced by other sources $j\neq i$.
Denote by $p^{i}\geq 0$ the transmission power of source $i$. 
The power of the useful received signal at the destination $i$ is proportional to
$p^{i}A^{i,i}$ where $A^{i,i}$ represents the channel gain between the $i$th source and the corresponding destination.
On the other hand, the level of interference endured by the destination $i$ is proportional to
$\sum_{j\neq i} p^{j}A^{j,i}$ where $A^{j,i}$ is the (positive) channel gain
between source $j$ and destination $i$.
As will be explained below, the reliability of the link between the $i$th source and its destination 
is usually expressed as an increasing function of the 
\emph{signal to interference-plus-noise ratio}, defined as $A^{i,i}p^i/(\sigma_i^2 + \sum_{j\neq i}A^{j,i}p^j)$.


We assume that there is
no Channel State Information at the Transmitter (no CSIT) \emph{i.e.}, all
channel gains are unknown at all transmitters.  However, we assume that the
destination associated with the $i$th source-destination pair
\begin{itemize}
\item knows the set of channel gains $\bs A^i:=(A^{1,i},\cdots,A^{N,i})^T$,
\item ignores all other channel gains $\bs A^j$ for $j\neq i$.
\end{itemize}
Figure~\ref{fig:ic} below illustrates the interference channel with $N=2$
transmit-destination pairs.
\begin{figure}[h]
  \centering
  \includegraphics[width=\columnwidth]{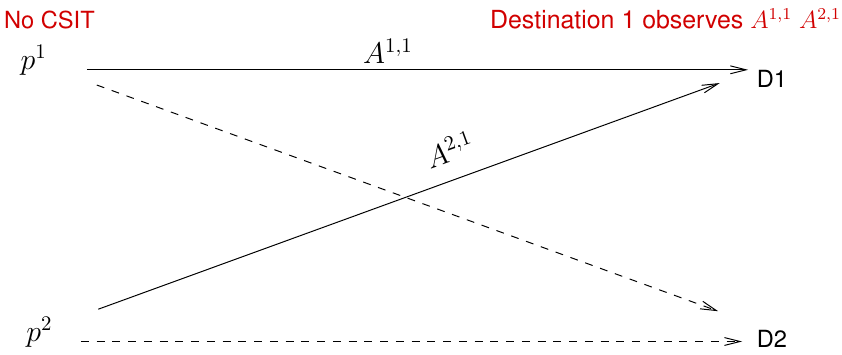}
  \caption{Example of a $2\times 2$ interference channel.}
  \label{fig:ic}
\end{figure}
We assume that $0\leq
p^{i}\leq \cP_i$ where $\cP_i$ is the maximum allowed power for user $i$.
Define $\theta = ({p^1},\cdots,{p^N})^T$ as the vector of all powers of all
users.  The aim is to select a relevant value for parameter $\theta$.  We assume
that destinations are able to communicate according to an underlying connected
graph.  The proposed algorithm works as follows.
\begin{enumerate}
\item In a first step, the set of destination nodes cooperate and jointly search
  for a relevant global power allocation $\theta$. The desired vector $\theta$
  corresponds to a local minimizer of an optimization problem which will be made
  clear below.
\item Once an agreement is found on the power allocation vector $\theta$, each
  destination $i$ provides its own source with the corresponding power $p^i$
  using a dedicated channel.
\end{enumerate}


Consider fixed deterministic channels. As a performance metric, consider
the error probability observed at each destination. Assuming for instance that
each transmitter uses a 4-QAM modulation, the error probability at the $i$th
destination is given by~\cite[Section 3.1]{tse:2005}:
\begin{equation}
  \label{eq:Pe}
  P_{e,i}(\theta,\bs A^i) := Q\left(\sqrt{\frac{A^{i,i}p^i}{\sigma_i^2 + \sum_{j\neq i}A^{j,i}p^j}}\right)\ .
\end{equation}
where $Q(x)=(\sqrt{2\pi})^{-1}\int_x^\infty e^{-t^2/2}dt$.
We investigate the following minimization problem:
\begin{equation}
  \min_{\theta\in G} \sum_{i=1}^N \beta_i\, P_{e,i}(\theta,\bs A^i)
  \label{eq:sumpe}
\end{equation}
where $\beta_i$ is an arbitrary positive deterministic weight known only by
agent $i$ and where $G:=\{(p^1,\cdots,p^N)\in\bR^N : \forall i=1,\cdots,N,\,
0\leq p^i\leq \cP^i\}$.  The above optimization problem is non-convex.  Note
that, utility functions~(\ref{eq:Pe}) can of course be replaced by any other
continuously differentiable functions of the signal-to-interference-plus-noise
ratio without changing the results of this section.
Section~\ref{sec:model} suggests the following deterministic distributed
gradient algorithm.  Each user $i$ has an estimate $\theta_{n,i}$ of the
\emph{whole} vector $\theta$ at the $n$th iteration.  Here, we stress the fact
that a given user has not only an estimate of its own power
allocation $p^i$, but has also an estimate of what should be the power
allocation of other users $j\neq i$. Denote by
$\thn=(\theta_{n,1}^T,\cdots,\theta_{n,N}^T)^T$ the vector of size $N^2$ which
gathers all local estimates.  Denote by $\bs A:=((\bs A^1)^T,\cdots,(\bs
A^N)^T)^T$ the vector which gathers all $N^2$ channel gains.  The distributed
algorithm writes:
\begin{equation}
  \label{eq:deterministicgradient}
  \thn = (W_n\otimes I_d)P_{G^{N}}\left[\thnmu + \gamma_n\,\Upsilon(\thnmu;\bs A) \right]
\end{equation}
where for any $\bth=(\theta_1^T,\cdots,\theta_N^T)^T$ in $\bR^{N^2}$, we set
$$
\Upsilon(\bth;\bs A)\!=\!(\beta_1 \nabla_\theta P_{e,1}(\theta_1;\bs
A^1)^T,\cdots,\beta_N \nabla_\theta P_{e,N}(\theta_N;\bs A^N)^T)^T
$$
and where $\nabla_\theta$ is the gradient operator with respect to the first
argument $\theta$ of $P_{e,i}(\theta,A^i)$.
\begin{coro}
  Under the stated assumptions on the sequences $(\gamma_n)_{n\geq 1}$ and
  $(W_n)_{n\geq 1}$, the algorithm~(\ref{eq:deterministicgradient}) is such that
  sequence $(\athn)_{n\geq 1}$ converges to the set of KKT points
  of~(\ref{eq:sumpe}).
\end{coro}

\begin{remark}
In many situations, the channel gains are random and rapidly
time-varying. In this case, it is more realistic to assume that each destination
$i$ observes a sequence of random i.i.d. channel gains $(\bs A^i_n)_{n\geq 1}$.  The
algorithm~(\ref{eq:deterministicgradient}) can be extended to this context without difficulty.
This yields an algorithm which for solving the following optimization problem:
\begin{equation}
  \min_{\theta\in G} \sum_{i=1}^N \beta_i\, \bE[P_{e,i}(\theta,\bs A^i_n)]\ .
  \label{eq:sumpealea}
\end{equation}
\end{remark}

\section{Numerical Results}
\label{sec:simu}

\subsection{Scenario \#1}

As a benchmark, we address the convex optimization scenario formulated
in~\cite{nedic:tac-2011}.  Define $G\subset \bR^2$ as the unit disk in $\bR^2$
centered at the origin.  Consider the minimization of $\sum_{i=1}^N f_i(\theta)$
w.r.t. $\theta\in G$, where for any $i=1,\cdots,N$,
$f_i(\theta):=\bE[(R_i-s_i^T\theta)^2]$. Here, $(R_1,\cdots,R_N)$ is a
collection of i.i.d. real Gaussian distributed random variables with mean $0.5$
and unit variance, and $(s_1,\cdots,s_N)$ is a collection of deterministic
elements of $\bR^2$.  The number of agents is set as $N=10$ or $N=50$.  We used
different graphs: the complete graph where any agent is connected to all
other agent, and the cycle.  We evaluate the performance of both pairwise and
broadcast algorithms described in Section~\ref{sec:gossip}.  The weighting
coefficient $\beta$ used to compute the average is set to $\beta=0.5$.  As for
comparison, we also evaluate the performance of the broadcast-based algorithm
of~\cite{nedic:tac-2011}.  The common point between the algorithm
of~\cite{nedic:tac-2011} and the broadcast algorithm described in
Section~\ref{sec:gossip} is that they both rely on the broadcast gossip scheme
of~\cite{aysal:2009} but the core of the algorithms is rather different as
explained in Section~\ref{sec:intro}. In order to distinguish both broadcast
algorithms, we will designate the algorithm of~\cite{nedic:tac-2011} as the
\emph{broadcast algorithm with sleeping phases}, refering to the fact that each
agent does not update its estimates as long as it is not the recipient of a
message. On the other hand, we refer to the broadcast algorithm of
Section~\ref{sec:gossip} as the \emph{broadcast algorithm without sleeping
  phases}.

It is worth remarking that a fair comparison between different
stochastic approximation algorithms is generally a delicate task,
because the behavior of each particular algorithm is sensitive to the
choice of the stepsize. In this paragraph, we simply set $\gamma_n =
\gamma_0/n^\xi$ for all $n$, where $\gamma_0>0$ and $0.5<\xi\leq 1$ are parameters chosen
in an ad-hoc fashion. More degrees of freedom are of course possible
when choosing $\gamma_n$, but a complete discussion would be out the scope
of this paper.  Recall that the algorithm of~\cite{nedic:tac-2011}
requires a more specific choice of the stepsize which 
solely depends on the initial step. We shall denote by $\gamma_0^s$ the latter initial
stepsize used with the algorithm~\cite{nedic:tac-2011}, where the upperscript $s$
stands for \emph{sleeping phases}.

For each algorithm, we evaluate the deviation of the estimates from the global minimizer
$\theta_\star$:
$$
\Delta_n :=\left(\frac 1N \sum_{i=1}^N |\theta_{n,i}-\theta_\star|^2\right)^{1/2}\ .
$$
Note that $\Delta_n$ depends on the parameters $(s_1,\cdots,s_N)$.
We consider 50 Monte-Carlo runs, each of them consisting of 10000 iterations
of each algorithm. For each run, we randomly select the parameters 
$(s_1,\cdots,s_N)$ according to the uniform distribution on the unit disk $G$.
The $k$th Monte-Carlo run yields a sequence $(\Delta_n^{(k)} : 1\leq n\leq 10000)$ for each algorithm.

Figure~\ref{fig:toy}  represents
the average deviation $(1/50)\sum_{k=1}^{50}\Delta_n^{(k)}$ as a function of the number $n$ of iterations.
In Figure~\ref{fig:toy}(a), we set $N=50$ and the graph is a cycle.
In Figure~\ref{fig:toy}(b), we set $N=10$ and the graph is a complete graph.
\begin{figure}[h]
  \centering
 \includegraphics[width=9cm]{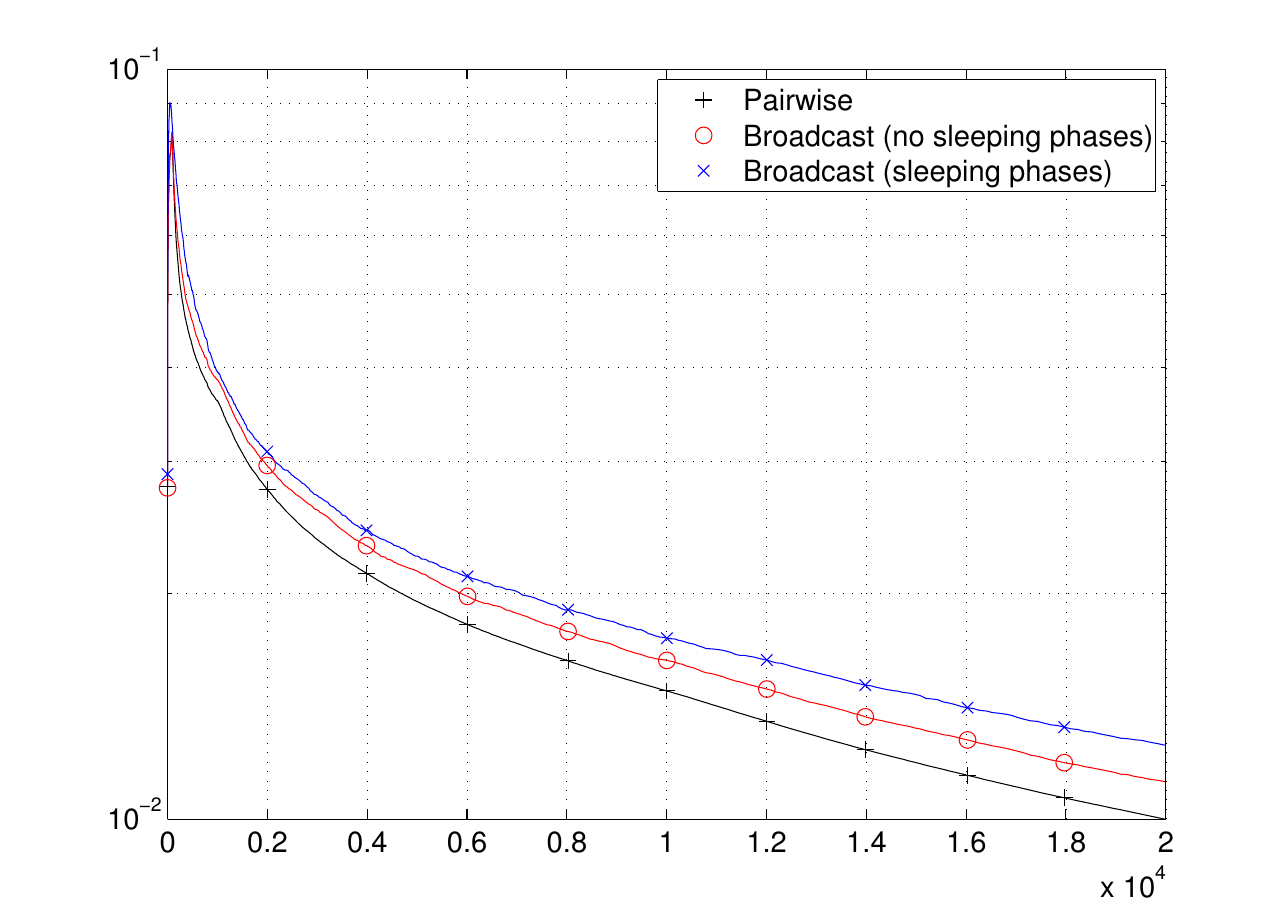} \\    {\small $(a)$}\\
\includegraphics[width=9cm,height=6.7cm]{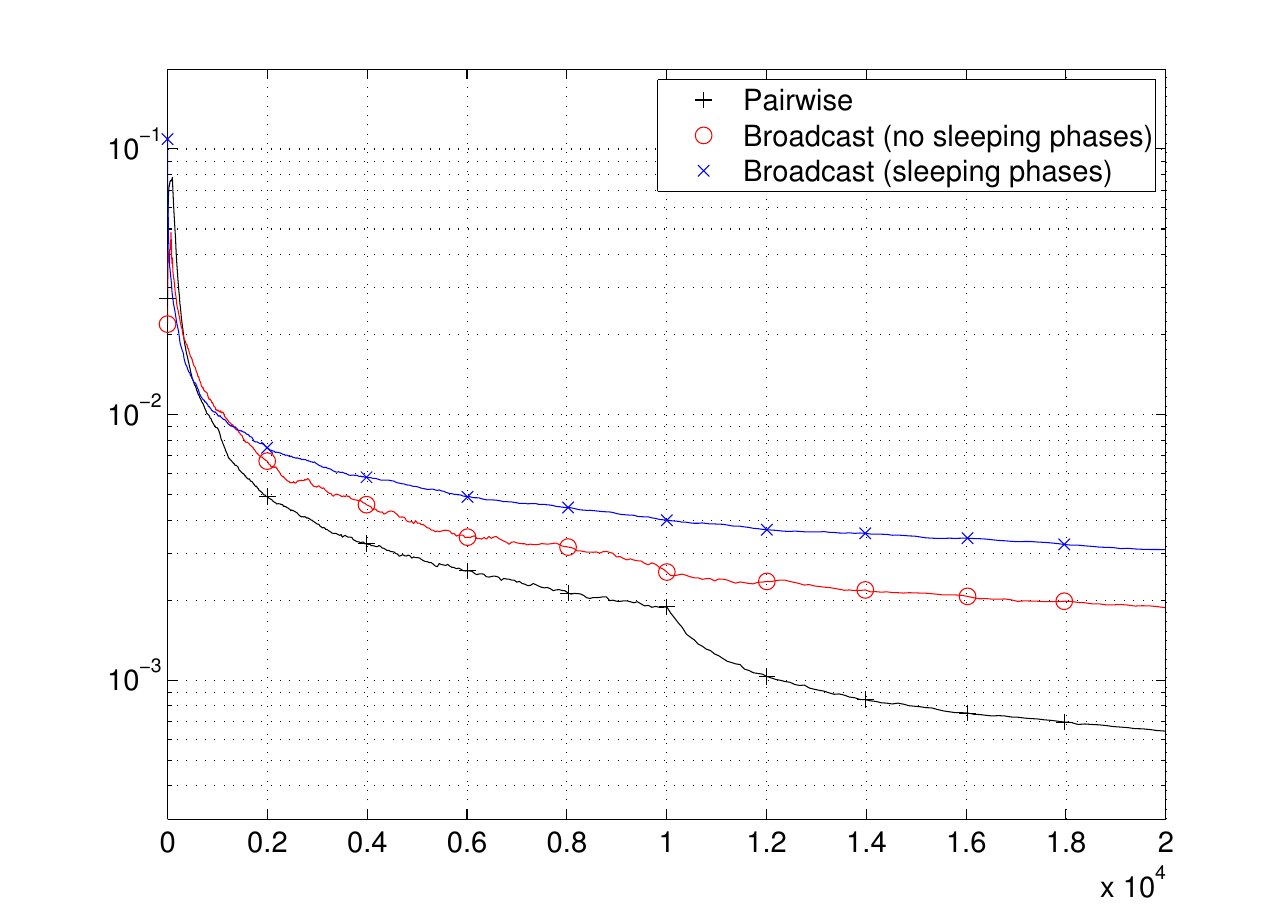}\\  {\small $(b)$}
  \caption{Average deviation as a function of the number of iterations 
    $(a)$ Cycle, $N=50$, $\gamma_n=0.1/n^{0.9}$, $\gamma_0^s=0.1$  - $(b)$ Complete graph, $N=10$, $\gamma_n=1/n^{0.9}$, $\gamma_0^s=5$}
  \label{fig:toy}
\end{figure}
It is worth noting that the pairwise gossip algorithm behaved at least as well
as both broadcast based algorithms in our experiments. This fact might seem
surprising at first glance. Indeed, in the framework of average consensus
\emph{i.e.}, when the aim is not to optimize an objective function but simply to
compute an average in a distributed fashion~\cite{boyd:2006}, the broadcast
gossip algorithm of~\cite{aysal:2009} is known to {\sl i)} reach a consensus
faster than the pairwise algorithm of~\cite{boyd:2006} and {\sl ii)} fail to
converge to the desired value. In the context of distributed optimization, a
different phenomenon happens: our theoretical analysis showed that broadcast
based optimizers do converge to the desired value. However, in this example,
there is no clear gain in using a broadcast-based algorithm. Convergence has
been established in Theorem~\ref{the:cvc}.Appendix~\ref{app:average} reveals
that part of the perturbation (denoted $\xi_n$) is due to the fact that
$\un^TW_n\neq \un^T$ (see the term $\xi_n^{(1)}$ at
equation~(\ref{eq:xinu})). This part of the perturbation is clearly zero when
$W_n$ is doubly stochastic. This is the case for the pairwise algorithm, but not
for the broadcast algorithm.

As a conclusion, there should be interesting comparisons to make between
pairwise optimizer and the broadcast ones.

\subsection{Scenario \#2}

Consider the distributed power allocation algorithm of Section~\ref{sec:appli}.
In order to validate the proposed algorithm, we study the 2$\times $2
interference channels shown in Figure~\ref{fig:ic}.  As a toy but revealing
example, first assume fixed channel gains chosen as $A^{1,1}=A^{2,2}=2$,
$A^{1,2}=A^{2,1}=1$. The noise variance is equal to $\sigma_1^2=\sigma_2^2=0.1$.
The powers $p^1$ and $p^2$ of the users must not exceed a maximum power of
$\cP_1=\cP_2=10$. The aim is to minimize the weighted sum of the error
probabilities as in~(\ref{eq:sumpe}) where $\beta_1=2/3$,
$\beta_2=1/3$. Strictly speaking, we actually implement a distributed gradient
descent w.r.t. to the parameter vector $\theta$ in log-scale in order to avoid
slow convergence.
\begin{figure}[h]
  \centering
  \includegraphics[width=9cm]{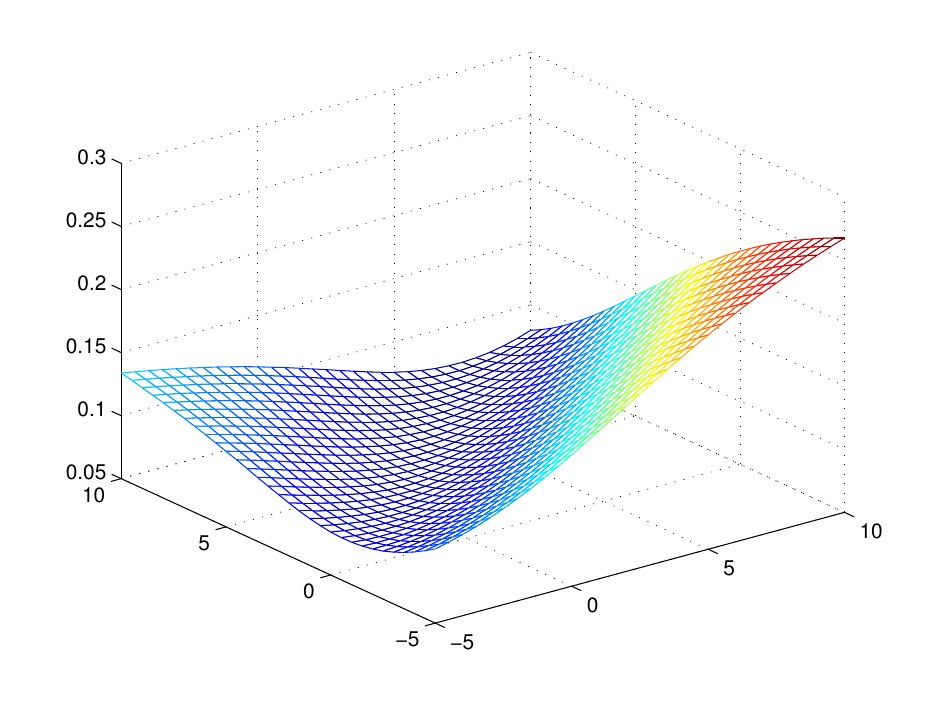}\\    {\small $(a)$}\\
\includegraphics[width=9cm]{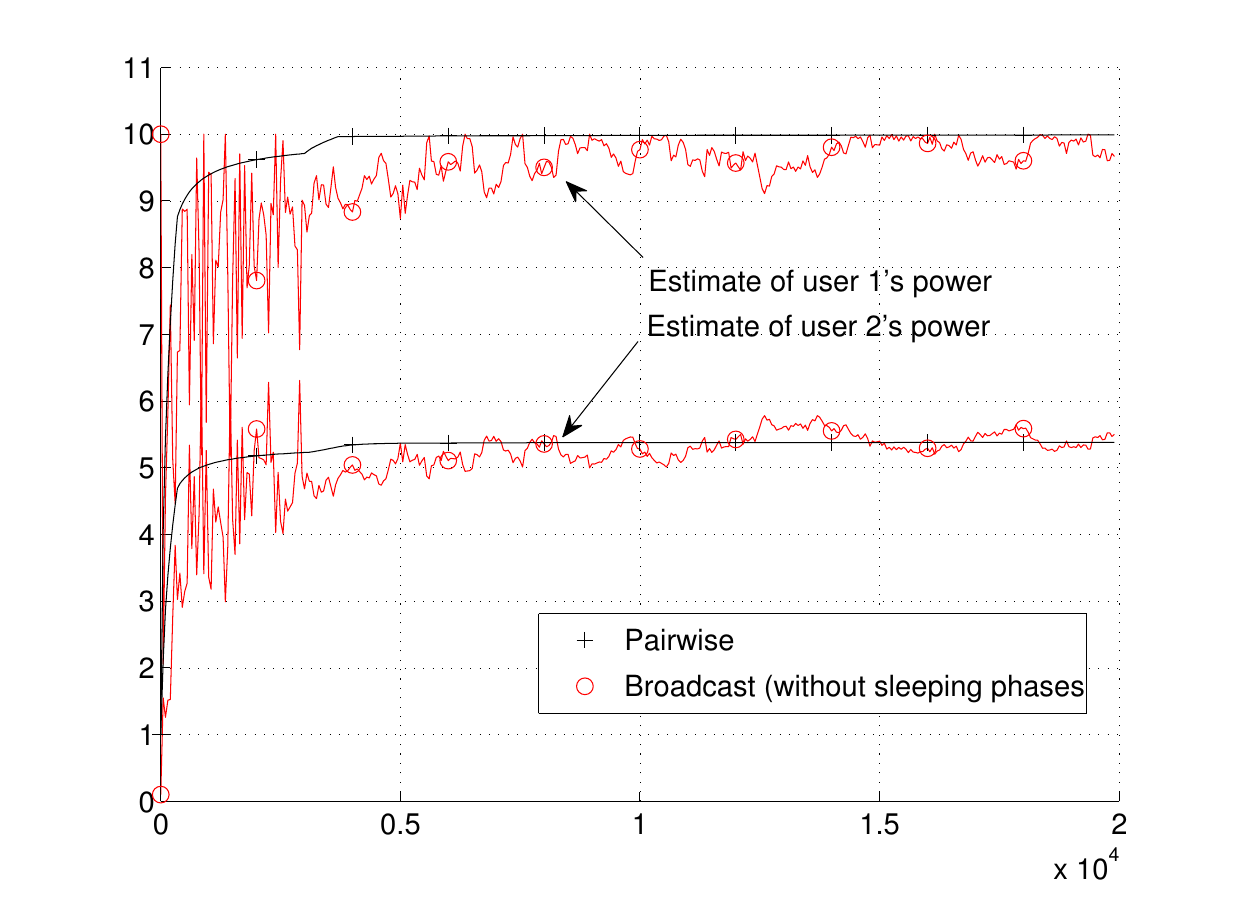}\\    {\small $(b)$}
  \caption{(a) Weighted sum of error probabilities for $N=2$ as a function of
    the powers $p^1$ and $p^2$ in dB - Fixed Determinitic channels -
    $A^{1,1}=A^{2,2}=2$, $A^{1,2}=A^{2,1}=1$ - $\beta_1=2/3$, $\beta_2=1/3$ -
    $\sigma_1^2=\sigma_2^2=0.1$ - $\cP_1=\cP_2=10$. The minimum is achieved at
    point $(p^1,p^2)=(10,5.4)$. (b) First agent's estimates of $p^1$ and $p^2$ as a function of the
    number of iterations - $\gamma_n=200/n^{0.7}$ for $n\leq 3000$ - $\gamma_n=30/n^{0.7}$ for $n> 3000$.}
  \label{fig:coutalloc}
\end{figure}
Figure~\ref{fig:coutalloc}(a) represents the objective function~(\ref{eq:sumpe})
w.r.t. $(p^1,p^2)$ in dB (the $x$-axis and $y$-axis are $10\log_{10} p^1$ and
$10\log_{10} p^2$ respectively).  On this example, there exists a unique minimum
achieved at point $(p^1,p^2)=(10,5.4)$.  Figure~\ref{fig:coutalloc}(b)
represents, on a single run, the trajectory of the estimates
$\theta_{n,1}=(p^1_{n,1},p^2_{n,2})$ of the \emph{first} agent as a function of
the number of iterations.
  We compare the pairwise and the broadcast gossip schemes. Note that we only plot
the result for the broadcast scheme without sleeping phase, as we observed slow
convergence of the algorithm of~\cite{nedic:tac-2011} on this particular
example.  The two upper curves represent the estimate of power $p_1$ (using a
pairwise and a broadcast scheme respectively) while the two lower curves
represent the estimate of power $p_2$.  Each algorithm converges to the desired
value $(10,5.4)$. However, the convergence curve is rather smooth in the
pairwise case, and is more erratic in the broadcast case.  Indeed, matrices
$W_n$ are non doubly stochastic in the broadcast scheme.  As already explained
above, non doubly stochastic matrices introduce an artificial noise term which
is the main cause of the erratic shape of the trajectory.

We finally provide numerical results in the case where channel gains are random
and time-varying.  We assume Rician fading~\cite[Section 2.4]{tse:2005}.  For
any $n$, we set $\bE A^{1,1}_n=\bE A^{2,2}_n=2$, $\bE A^{1,2}_n=\bE
A^{2,1}_n=1$.  The variance of each channel gain is 0.5. Components of $\bs A_n$
are assumed independent.
\begin{figure}[h]
  \centering
  \includegraphics[width=9cm]{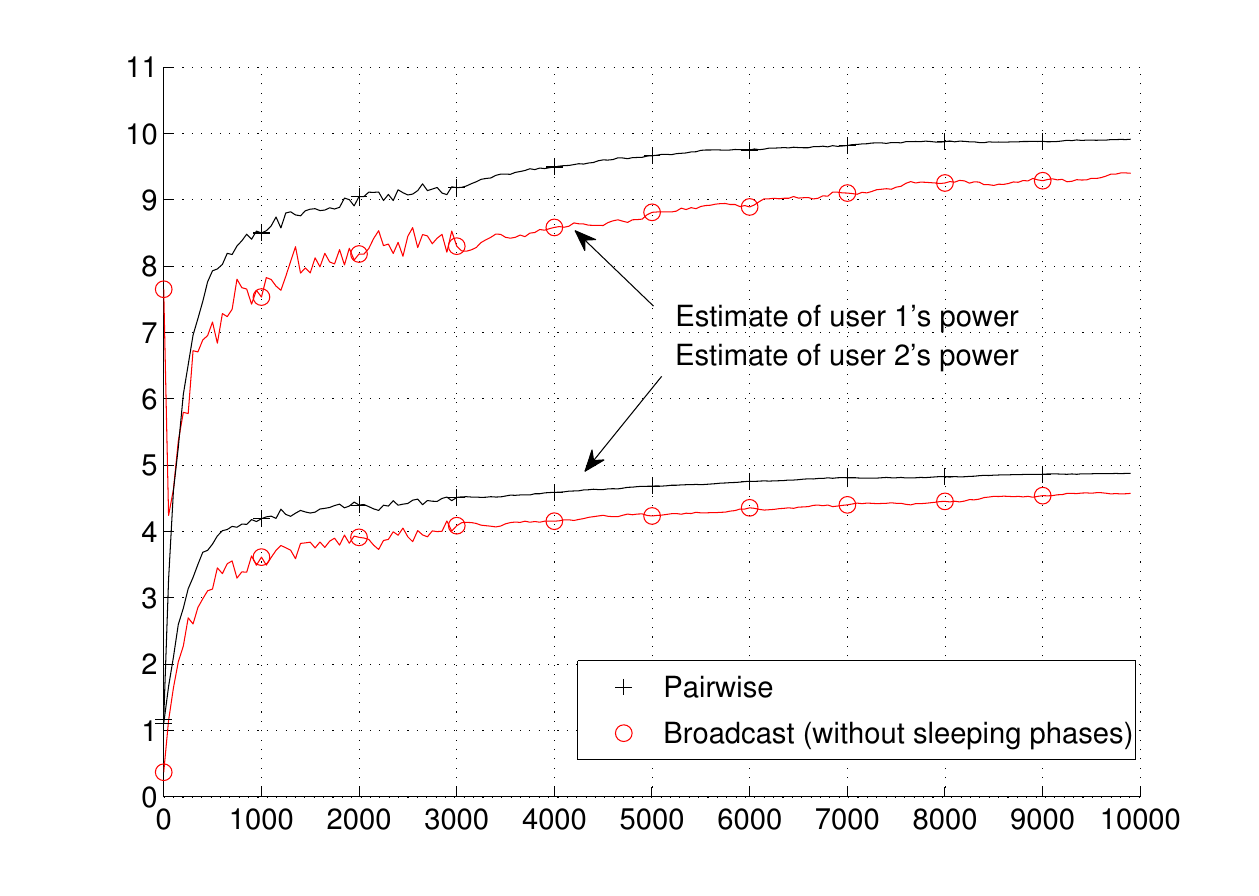}

  \caption{Powers $p^1$ and $p^2$ as a function of the number of iterations,
    averaged w.r.t. 50 Monte-Carlo runs - $\gamma_n=200/n^{0.7}$ for $n\leq 3000$ - $\gamma_n=30/n^{0.7}$ for $n> 3000$.}
  \label{fig:convergencenoisy}
\end{figure}
Figure~\ref{fig:convergencenoisy} represents the average trajectory of the
estimates $\theta_{n,1}=(p^1_{n,1},p^2_{n,2})$ of the first agent as a function
of the number of iterations.  Trajectories have been averaged based on 50
Monte-Carlo runs. Once again, we observe convergence of the distributed
algorithms. The convergence is faster in the pairwise case.

\section{Conclusion}

We introduced a new framework for the analysis of a class of constrained
optimization algorithms for multi-agent systems.  The methodology uses recent
powerful results about dynamical systems which do not rely on the convexity of
the objective function, thus addressing a wider range of practical distributed
optimization problems.  Also, the proposed framework allows to alleviate the
common assumption of double-stochasticity of the gossip matrices, and therefore
encompasses the natural broadcast gossip scheme.  The algorithm has been proved
to converge to a consensus.  The interpolated process of average estimates is
proved to be a perturbed solution to a differential variational inequality,
w.p.1. As a consequence, the average estimate converges almost surely to the set
of KKT points.

\section*{Acknowledgments}

We are grateful to the anonymous Reviewers for their useful comments.
We thank Philippe Ciblat for relevant remarks regarding Section V, as well as
Eric Moulines and Filippo Santambrogio for fruitful discussions.

\appendices

\section{Proof of Proposition~\ref{prop:average}}
\label{app:average}

From~(\ref{eq:algo}) and Assumption~\ref{hyp:model}(e), it is straightforward to
show that the decomposition~(\ref{eq:rmc}) holds if one sets:
$$
r_n = \nabla f(\athnmu)-\frac 1N\sum_{i=1}^N\nabla f_i(\theta_{n-1,i})
$$
and $\xi_n=\xi_{n}^{(1)}+\xi_n^{(2)}$ where:
\begin{eqnarray}
  \label{eq:xinu}
  \xi_n^{(1)}\!\!\!\! &:=&\!\!\!\!\! \frac 1{\gamma_n}\left(\!\frac {\un^TW_n-\un^T}N\!\otimes\!
    I_d\!\right)P_{G^N}\!\left[\thnmu+\gamma_n\bY_n\right]\\
  \xi_n^{(2)}\!\!\!\! &:=& \!\!\!\la\bs Z_n\ra-\bE[\la\bs Z_n\ra\,|\cF\nmu] \nonumber
\end{eqnarray}
where $\bs Z_n$ is given by (\ref{eq:Zn}).
We first prove that $r_n$ tends to zero. Remark that:
$$
|r_n|\leq \frac 1N\sum_{i=1}^N |\nabla f_i(\athnmu) - \nabla
f_i(\theta_{n-1,i})|\ .
$$
Each gradient $\nabla f_i$ is continuous, and thus uniformly continuous on the
compact set $G$.  By Lemma~\ref{lem:agreement}, $|\athnmu - \theta_{n-1,i}|$
converges to zero a.s. for any $i$.  Therefore, $r_n$ converges a.s. to zero.
To prove Proposition~\ref{prop:average}, it is thus sufficient to show that
$\sup_{k\geq n}\left|\sum_{\ell=n}^k \gamma_\ell\xi_{\ell}^{(j)}\right|$ tends
to $0$ when $n\to\infty$ for $j=1,2$.  First, consider $\xi_n^{(1)}$. Recalling
that $W_n$ is row-stochastic, it follows that $\{(\un^TW_n-\un^T)\otimes
I_d\}J=0$. Thus, one may write:
\begin{equation*}
  \gamma_n \xi_n^{(1)} = \left(\frac {\un^TW_n-\un^T}N\otimes
    I_d\right)(\othnmu+\gamma_n\bsZ_n)\ ,
\end{equation*}
where the random vector $\bsZ_n$ is given by~(\ref{eq:Zn}).  Define
$M_n:=\sum_{k=1}^n \gamma_k\xi_{k}^{(1)}$. It is straightforward to show that
$M_n$ is a martingale adapted to $(\cF_n)_{n\geq 1}$. Indeed, by
Assumption~\ref{hyp:model}(c), $W_n$ and $\bsZ_n$ are independent conditioned
on $\cF\nmu$. Therefore:
\begin{multline*}
  \bE[\gamma_n \xi_n^{(1)}\,|\cF\nmu] = \left(\frac {\un^T\bE
      (W_n)-\un^T}N\otimes I_d\right)\\ \times (\othnmu+\gamma_n\bE[\bsZ_n|\cF\nmu])\ =0
\end{multline*}
because  $\un^T\bE (W_n)=\un^T$ by Assumption~\ref{hyp:model}(a).  We
derive:
\begin{eqnarray*} 
  \bE |M_n|^2 \!\! \!\!\!\! &=&\!\!\!\! \sum_{k=1}^n \bE\left[ \left|\left(\frac
        {\un^TW_k-\un^T}N\otimes I_d\right)(\othkmu+\gamma_k\bsZ_k) \right|^2\right] \\
  &=&\!\!\!\!  \sum_{k=1}^n \bE\left[ (\othkmu+\gamma_k\bsZ_k)^T\cV_k (\othkmu+\gamma_k\bsZ_k)\right]\\
  &=&\!\!\!\!  \sum_{k=1}^n \bE\left[ (\othkmu+\gamma_k\bsZ_k)^T \bE(\cV_k)(\othkmu+\gamma_k\bsZ_k)\right]
\end{eqnarray*}
where $\cV_k = \left(\frac
      {(W_k^T\un-\un)(\un^TW_k-\un^T)}{N^2}\otimes I_d\right)$
and where the last equality is obtained by integrating the inner terms w.r.t. $W_k$ only.
Remark that
$\bE[(W_k^T\un-\un)(\un^TW_k-\un^T)]=\bE[W_k^T\un\un^TW_k]-\un\un^T$. As the
spectral radius of matrix $\bE[W_k^T\un\un^TW_k]$ is uniformly bounded, there
exists a constant $C'>0$ such that:
\begin{eqnarray*}
  \bE |M_n|^2 \!\!\!\! &\leq&  \!\!\!\!C'\,\sum_{k=1}^\infty \bE\left[ |\Jo\thkmu+\gamma_k\bsZ_k|^2\right]\\
  &\leq& \!\!\!\! 2 C'\,\sum_{k=1}^\infty \bE\left[ |\Jo\thkmu|^2\right]+2
  C'\,\sum_{k=1}^\infty\gamma_k^2\,\bE\left[ |\bsZ_k|^2\right].
\end{eqnarray*}
By Lemma~\ref{lem:agreement}, the first term in the righthand side of the above
inequality is finite. Recalling that $|\bsZ_k|\leq |\bY_k|$, we deduce from
Assumption~\ref{hyp:model}(f) that $\bE|\bsZ_k|^2$ is uniformly bounded. As
$\sum_k\gamma_k^2<\infty$, we conclude that $\sup_n\bE |M_n|^2<\infty$. This
implies that the martingale converges a.s. to a finite random variable
$M_\infty$. Thus, for any $k\geq n$,
$$
\left|\sum_{\ell=n}^k \gamma_\ell\xi_{\ell}^{(1)}\right| =
\left|(M_{k}-M_{\infty})-(M_{n-1}-M_{\infty})\right|\ .
$$
Thus, $\sup_{k\geq n}\left|\sum_{\ell=n}^k \gamma_\ell\xi_{\ell}^{(1)}\right|$ tends a.s. to zero as $n\to\infty$.

We now study $\xi_n^{(2)}$. Clearly, $\xi_n^{(2)}$ is a martingale difference
noise sequence. Therefore,
\begin{eqnarray*}
  \bE\left[\left|\sum_{k=1}^n \gamma_k \xi_k^{(2)}\right|^2\right]
  &=&\sum_{k=1}^n \gamma_k^2\, \bE\left[\left|\la 
\bs Z_k-\bE[\bs Z_k\,|\cF_{k-1}]
\ra\right|^2\right]\\
  &\leq& \sum_{k=1}^\infty \gamma_k^2\, \bE\left[\left|\bs Z_k\right|^2\right] \\
&\leq&
  \sup_{n\geq 1}\bE\left|\bY_n\right|^2 \sum_{k=1}^\infty \gamma_k^2 \ , 
\end{eqnarray*}
where we used the fact that $|\bY_k|\leq |\bs Z_k|$ for any $k$.
Note that $\bE[\left|\bY_n\right|^2|\cF\nmu]$ is uniformly bounded with $\sup_{\bth} \int |\by|^2\sd\mu_{\bth}$
by Assumption~\ref{hyp:model}(c). The latter constant is finite by Assumption~\ref{hyp:model}(f).  
Therefore, $\sup_{n\geq 1}\bE\left|\bY_n\right|^2<\infty$.
Since $\sum_{k} \gamma_k^2<\infty$, we conclude that $\sup_{k\geq n}\left|\sum_{\ell=n}^k \gamma_\ell\xi_{\ell}^{(2)}\right|$
tends to zero using the same arguments.  This completes the proof of
Proposition~\ref{prop:average}.

\section{Proof of Lemma~\ref{lem:lambda}}
\label{app:lambda}

Let us define $Q(\theta)$ as the matrix $Q(\theta):=[\nabla q_j(\theta)]_{j\in
  A(\theta)}$.  Denote by $\Lambda_1(\theta)$ the smallest eigenvalue of
$Q(\theta)^T Q(\theta)$. We first show that $\Lambda_1(\theta)$ is lower
semicontinuous \emph{i.e.}, for a sequence $\theta_n\in G$ converging to
$\theta_*\in G$:
\begin{equation}
  \label{eq:lsc}
  \Lambda_1(\theta_*)\leq \liminf_n \Lambda_1(\theta_n)\;.
\end{equation}
Continuity of all functions $q_j$ ensures that $A(\theta)$ is upper
semicontinuous, \emph{i.e.}, for any $\theta$ in a neighborhood of $\theta_*$,
$A(\theta)\subset A(\theta_*)$. Hence, for $n$ large enough $A(\theta_n)\subset
A(\theta_*)$. Denote by $\tilde Q(\theta_n)$ the matrix $d\times p$:
\[
\tilde Q(\theta_n) = [\nabla q_j(\theta_n) \bs 1_{A(\theta_*)}(j)]_{j= 1\cdots
  p}
\]
where $\bs 1_A$ stands for the indicator function of set $A$. There exists a
sequence of $p\times  1$ vectors $\tilde v_n$ with unit norm such that $|\tilde
Q(\theta_n)\tilde v_n|^2 = \Lambda_1(\theta_n)$ and $\tilde v_n(j) = 0$ if
$j\not\in A(\theta_n)$. Since $\tilde v_n$ has unit norm, one can extract a
converging subsequence $\tilde v_{\phi(n)}$ towards a unit norm $p \times  1$ vector
$\tilde v_*$ such that $|\tilde Q(\theta_{\phi(n)})\tilde v_{\phi(n)}|^2$
converges to $\liminf_n\Lambda_1(\theta_n)$. Using the inclusion
$A(\theta_n)\subset A(\theta_*)$ one has $\tilde v_*(j) = 0$ when $j\notin
A(\theta_*)$. Moreover, under Assumption~\ref{hyp:KT}(b), functions $\nabla q_j$
are continuous, which implies that $\tilde Q(\theta_n)$ converges to
$Q(\theta_*)$, hence vector $\tilde v_*$ satisfies $|\tilde Q(\theta_*)\tilde
v_*|^2 = \lim_n |\tilde Q(\theta_{\phi(n)})\tilde v_{\phi(n)}|^2$.  Since
$\tilde v_*(j) = 0$ when $j\notin A(\theta_*)$, there exists a vector $v_*$ such
that $|Q(\theta_*)v_*|^2 = |\tilde Q(\theta_*)\tilde v_*|^2$. Hence
\[
\Lambda_1(\theta_*)\leq |Q(\theta_*)v_*|^2 = |\tilde Q(\theta_*)\tilde v_*|^2 =
\liminf_n\Lambda_1(\theta_n)\;.
\]
This proves (\ref{eq:lsc}).  Under Assumption~\ref{hyp:KT}(a) $G$ is a compact
set, so lower semicontinuity of $\Lambda_1(\theta)$ ensures that $\Lambda_1$
reaches its minimum $m>0$ ($m=0$ would contradict Assumption~\ref{hyp:KT}(c)).
Now, let us denote by $\lambda:=(\lambda_j(\theta,\gamma,y))^T_{j\in
  A(\theta+\gamma y)}$ and $v := \frac1{\gamma}(\theta+\gamma y-P_G(\theta+\gamma y))$. One has
\[
\lambda \!\!= \!\!\left(Q\big(P_G(\theta+\gamma y)\big)^TQ\big(P_G(\theta+\gamma
  y)\big)\right)^{-1}\!\!\! Q\left(P_G(\theta+\gamma y)\right)^T\!\!v.
\]
Hence $|\lambda| \leq \Lambda_1^{-1}(P_G(\theta+\gamma y))|Q(P_G(\theta+\gamma
y))^T v|$.  Continuity of $\nabla q_j$ and compactness of $G$ ensure the
existence of $L>0$ such that: $|Q(P_G(\theta+\gamma y))^T v|\leq L|v|$. To
conclude, remark that $|v|\leq |y|$ so $|\lambda|\leq \frac Lm|y|$. Hence,
\begin{multline*}
  \int\lambda_j(\theta_i,\gamma,y_i)^2\sd\mu_\bth(y_1,\dots,y_N) \\ \leq \left(\frac
  Lm\right)^2\int|y_i|^2\sd\mu_\bth(y_1,\dots,y_N)<\infty\;.
\end{multline*}

\section{Proof of Lemma~\ref{lem:Aepsilon}}
\label{app:Aepsilon}

Define constant $M_1$ as the supremum in equation~(\ref{eq:suplambda}): $0\leq
M_1<\infty$ by Lemma~\ref{lem:lambda}. We set $M_2=\sup_{\theta\in G,j=1\cdots
  p}|\nabla q_j(\theta)|$.  Define for any $x> 0$, $M(x) = \sup_{\bth\in G^N}\int\bs 1_{\{|y|>x^{-1/2}\}}\sd\mu_{\bs\theta}(y)$ and:
$$
\epsilon(x) = \sqrt{x}+x+ 2p\sqrt{M_1}\phi(\sqrt{x}+x) +
2p\sqrt{M_1}M_2\left(M(x)\right)^{1/2}
$$
where we recall the definition~(\ref{eq:phi}) of $\phi$.  Using the fact that
$\phi(x)$ tends to zero as $x\downarrow 0$ and using the tightness of the family
$(\mu_{\bth})_{\bth\in G^N}$ (which is a consequence of
Assumption~\ref{hyp:model}(f)), it is straightforward to show that $\epsilon(x)$
tends to zero as $x\downarrow 0$.  
Set $y_{1:N} :=y_1,\dots,y_N$.
We decompose $-g_\gamma(\bth)$ as
$-g_\gamma(\bth)=:s_\gamma(\bth)+t_\gamma(\bth)+u_\gamma(\bth)$ where:
\begin{multline*}
    s_\gamma(\bth) = \frac 1{N}\sum_{i=1}^N\int \sum_{j\in A(P_{G}[\theta_i+\gamma y_i])} \lambda_j(\theta_i,\gamma,y_i)\bs 1_{|y_i|\leq 1/\sqrt{\gamma}}  \\ \times \nabla q_j(\ath)d\mu_\bth(y_{1:N})
\end{multline*}
\begin{multline*}
  t_\gamma(\bth) = \frac 1{N}\sum_{i=1}^N\int \sum_{j\in A(P_{G}[\theta_i+\gamma y_i])} \lambda_j(\theta_i,\gamma,y_i) \bs 1_{|y_i|>1/\sqrt{\gamma}}  \\ \times \nabla q_j(P_{G}[\theta_i+\gamma y_i])\sd\mu_\bth(y_{1:N})
\end{multline*}
\begin{multline*}
  u_\gamma(\bth) = \frac 1{N}\sum_{i=1}^N\int \sum_{j\in A(P_{G}[\theta_i+\gamma y_i])} \lambda_j(\theta_i,\gamma,y_i)\bs 1_{|y_i|\leq 1/\sqrt{\gamma}} \\ \times  (\nabla q_j(P_{G}[\theta_i+\gamma y_i])-\nabla q_j(\ath))\sd\mu_\bth(y_{1:N})\,.  
\end{multline*}
Consider first $s_\gamma(\bth)$. When the indicator $\bs 1_{|y_i|\leq
  1/\sqrt{\gamma}}$ is active (equal to one), inequality $|y_i|\leq
1/\sqrt{\gamma}$ holds true. In this case,
\begin{eqnarray*}
  |P_{G}[\theta_i+\gamma y_i]-\ath|&\leq& |P_{G}[\theta_i+\gamma y_i]-\theta_i|+|\theta_i-\ath| \\
&\leq&  |\gamma y_i| + |\oth| \leq \sqrt{\gamma}+ |\oth|\\
  &\leq& \epsilon(\gamma\lor |\oth|)\ .
\end{eqnarray*}
Therefore, as soon as $|y_i|\leq 1/\sqrt{\gamma}$, $A(P_{G}[\theta_i+\gamma
y_i])$ is included in the set $\cA(\ath,\epsilon(\gamma\lor |\oth|))$ where
$\cA$ is defined by~(\ref{eq:Aepsilon}). As a consequence,
$$
s_\gamma(\bth) = \!\!\!\!\!\!\!\! \sum_{j\in \cA(\ath,\epsilon(\gamma\lor
  |\oth|))}\!\!\!\!\!\!\!\!\alpha_j \nabla q_j(\ath)
$$
where $\alpha_j:= \frac
1{N}\sum_{i=1}^N\int \lambda_j(\theta_i,\gamma,y_i)\bs 1_{|y_i|\leq
  1/\sqrt{\gamma}}\sd\mu_\bth(y_{1:N})$.  By Jensen's inequality, $0\leq \alpha_j\leq \sqrt{M_1}$
for any~$j$.

Consider the second term $t_\gamma(\bth)$. It is straightforward to show from
triangle and Cauchy-Schwartz's inequalities that:
\begin{eqnarray*}
  |t_\gamma(\bth)|\!\!\!\!&\leq& \!\!\!\!\frac {M_2}{N}\sum_{i=1}^N\sum_{j=1}^p \Big(\int (\lambda_j(\theta_i,\gamma,y_i)^2\sd\mu_\bth(y_{1:N})) \Big)^{\frac 12}\\ &&\quad\quad\times \Big(\int(\bs 1_{|y_i|>1/\sqrt{\gamma}}\sd\mu_\bth(y_{1:N}))\Big)^{\frac 12}  \\
  &\leq&\!\!\!\! p\sqrt{M_1}M_2\left(M(\gamma)\right)^{1/2} \leq 0.5\, \epsilon(\gamma\lor |\oth| )\ .
\end{eqnarray*}
Finally, consider $u_\gamma(\bth)$:
\begin{eqnarray*}
  |u_\gamma(\bth)| \!\!\!\!&\leq &\!\!\!\! \frac 1{N}\sum_{i=1}^N \sum_{j=1}^p \int\lambda_j(\theta_i,\gamma,y_i)\bs 1_{|y_i|\leq 1/\sqrt{\gamma}}\\ &&\quad \times |\nabla q_j(P_{G}[\theta_i+\gamma y_i])-\nabla q_j(\ath))|\sd\mu_\bth(y_{1:N})\,.
\end{eqnarray*}
Note that $|\nabla q_j(P_{G}[\theta_i+\gamma y_i])-\nabla q_j(\ath))|\leq \phi(|P_{G}[\theta_i+\gamma y_i]-\ath|)$.
Again, we use the fact that $|P_{G}[\theta_i+\gamma y_i]-\ath|\leq \sqrt{\gamma}+ |\oth|$. As $\phi$ is non decreasing, 
it is clear that $\bs 1_{|y_i|\leq 1/\sqrt{\gamma}}\,  \phi(|P_{G}[\theta_i+\gamma y_i]-\ath|)$ is no larger than $\phi(\sqrt{\gamma}+ |\oth|)$.
Therefore,
$$
|u_\gamma(\bth)| \leq p\sqrt{M_1}\phi(\sqrt{\gamma}+ |\oth|) \leq 0.5\,
\epsilon(\gamma\lor |\oth| )\ .
$$
This completes the proof of Lemma~\ref{lem:Aepsilon}.

\section{Proof of Lemma~\ref{lem:AA}}
\label{app:AA}

For any $\epsilon\geq 0$, $j=1,\cdots,p$, define $\dG_j^\epsilon := \{\theta\in
G : \exists \theta'\in q_j^{-1}(\{0\}), |\theta'-\theta|\leq \epsilon\}$.  It is
useful to remark that $\dG_j^0 =q_j^{-1}(\{0\})\cap G$ is the set of points in
$G$ for which the $j$th constraint is active. In particular, that
$\dG_j^0\subset \dG_j^\epsilon$ for any $\epsilon\geq 0$.  Denote by $\sd_\sH$
the Hausdorff distance between sets. Define:
$$
\delta(\epsilon) = \max_{E\subset\{1,\cdots,p\}}\sd_\sH\left(\bigcap_{j\in E}
  \dG_j^\epsilon\,,\,\bigcap_{j\in E} \dG_j^0\right)\ .
$$
The key point is to show that $\lim_{\epsilon\downarrow 0}\delta(\epsilon)=0$.
By contradiction, assume that this is not the case.  Then there exists a
constant $c>0$ and a sequence $\epsilon_n\downarrow 0$ such that
$\delta(\epsilon_n)>c$ for each $n$. As there is a finite number of subsets of
$\{1,\cdots,p\}$, it is straightforward to show that there exists a certain
subset $E\subset \{1,\cdots,p\}$ such that for any $n\geq 1$,
\begin{equation}
  \label{eq:hausdorff}
  \sd_\sH\left(\bigcap_{j\in E} \dG_j^{\epsilon_n}\,,\,\bigcap_{j\in E} \dG_j^0\right)>c\ .
\end{equation}
First note that $\bigcap_{j\in E} \dG_j^{\epsilon_n}$ is nonempty. Indeed, if it
was empty, $\bigcap_{j\in E} \dG_j^0$ would be empty as well, so that the
Hausdorff distance in the lefthand side of~(\ref{eq:hausdorff}) would be
$\sd_\sH(\emptyset ,\emptyset )=0<c$.  Thus, for any $n\geq 1$, there exists
$\theta_n\in \bigcap_{j\in E} \dG_j^{\epsilon_n}$ such that
\begin{equation}
  \sd\left(\theta_n,\bigcap_{j\in E} \dG_j^0\right)>c\ .
  \label{eq:disttodG}
\end{equation}
The sequence $(\theta_n)_{n\geq 1}$ lies in the compact set $G$. Thus, there
exists a subsequence which converges to some point $\theta_\star\in G$. Without
loss of generality, we shall still denote this subsequence by $(\theta_n)_{n\geq
  1}$ in order to simplify the notations.  We thus consider that
$\lim_{n\to\infty} \theta_n =\theta_\star$.  We shall now prove that
$\theta_\star\in \bigcap_{j\in E} \dG_j^0$.  For any $n\geq 1$, $\theta_n\in
\dG_j^{\epsilon_n}$. Thus, there exists $\theta_n^{(j)}\in G$ such that
$q_j(\theta_n^{(j)})=0$ and $|\theta_n-\theta_n^{(j)}|\leq \epsilon_n$.  As
$q_j$ is convex, it is also Lipschitz on the compact set $G$. Denote by $K_j$
its Lipschitz constant on $G$:
$$
|q_j(\theta_n)|= |q_j(\theta_n)-q_j(\theta_n^{(j)})|\leq K_j \epsilon_n\ .
$$
Since $q_j$ is continuous and $\epsilon_n\downarrow 0$, it follows that
$q_j(\theta_\star)=0$.  Thus $\theta_\star\in \bigcap_{j\in E}
\dG_j^0$. Therefore, by~(\ref{eq:disttodG}), $|\theta_n-\theta_\star|>c$.  This
contradicts the fact that $(\theta_n)_{n\geq 1}$ converges to $\theta_\star$.
This proves that $\delta(\epsilon)$ tends to zero as $\epsilon\downarrow 0$.

It is useful to remark that, as a by-product of the above proof, we also
obtained the following result. Consider any set $E\subset\{1,\cdots,p\}$ and
assume that there exists a sequence $\epsilon_n\downarrow 0$ s.t.  for any
$n\geq 1$ there exists $\theta_n\in \bigcap_{j\in E} \dG_j^{\epsilon_n}$. Due to
the above arguments, any limit point of such a sequence $(\theta_n)_{n\geq 1}$
belongs to the set $\bigcap_{j\in E} \dG_j^0$ which is thus nonempty.  Let us
state this result the other way around: for any $E$ such that $\bigcap_{j\in E}
\dG_j^0=\emptyset$, there exists $\epsilon_0^E>0$ such that for
any~$\epsilon<\epsilon_0^E$, $\bigcap_{j\in E} \dG_j^\epsilon=\emptyset$.  We
set $\epsilon_0 = \min\{\epsilon_0^E : \bigcap_{j\in E} \dG_j^0=\emptyset\}$.

It remains to prove~(\ref{eq:AA}). Let $0<\epsilon<\epsilon_0$ and $\theta\in
G$. Trivially,
$$
\theta \in \bigcap_{j\in \cA(\theta,\epsilon)} \dG_j^\epsilon\ .
$$
As $\epsilon<\epsilon_0$, the set $\bigcap_{j\in \cA(\theta,\epsilon)} \dG_j^0$
is nonempty. There exists $\theta'$ in the latter set such that
$|\theta-\theta'|\leq \delta(\epsilon)$. By definition of $\theta'$,
$q_j(\theta')=0$ for any $j\in\cA(\theta,\epsilon)$. This proves that
$\cA(\theta,\epsilon)\subset A(\theta')$.

\bibliographystyle{IEEEtran}
\bibliography{biblio}

\begin{biographynophoto}
{Pascal Bianchi} received the M.Sc. degree of Sup{\'e}lec-Paris XI in 2000 and the Ph.D. degree of the University of Marne-la-Vall{\'e}e in 2003. 
From 2003 to 2009, he was an Associate Professor at the Telecommunication Department of Supélec.
In 2009, he joined the Statistics and Applications group at LTCI-Telecom ParisTech.
His current research interests are in the area of statistical signal processing for sensor networks.
They include decentralized detection, quantization, stochastic optimization, and applications of random matrix theory.
\end{biographynophoto}

\begin{biographynophoto}
  {J\'er\'emie Jakubowicz} received the MS degree (2004) and the PhD
  degree (2007) in applied mathematics from the Ecole Normale
  Sup\'erieure de Cachan. He was an Assistant Professor with T\'el\'ecom
  ParisTech. Since 2011 he has been an Assistant Professor with
  T\'el\'ecom SudParis and an Associate Researcher with the CNRS. His
  current research interests include distributed statistical signal
  processing, image processing and data mining.
\end{biographynophoto}

\end{document}